\documentclass[twoside,11pt]{article}
\usepackage[hyphens]{url}

\usepackage[preprint]{jmlr2e}
\setcitestyle{numbers,square,comma}
\hypersetup{
	colorlinks = true,
	hypertexnames=false
}

\usepackage{amsmath,mathtools}
\usepackage{bm}
\usepackage{xcolor}
\usepackage{slashbox}
\usepackage[para,online,flushleft]{threeparttable}
\usepackage{doi}

\newcommand{\dd}{\,\mathrm{d}}
\newcommand{\norm}[1]{\left\lVert#1\right\rVert}

\usepackage{lastpage}
\ShortHeadings{Simple Note on Subgaussians}{Li}
\firstpageno{1}

\begin{document}

\title{\Large A Simple Note on the Basic Properties of\\Subgaussian Random Variables}

\author{\name Yang Li \email liyang@g.ecc.u-tokyo.ac.jp \\
       \addr International Research Center for Neurointelligence\\
       the University of Tokyo\\
       7-3-1 Hongo, Bunkyo-ku, Tokyo 113-0033, Japan}

%\editor{My editor}

\maketitle

\begin{abstract}%   <- trailing '%' for backward compatibility of .sty file
\indent This note provides a basic description of subgaussianity,
by defining $(\sigma, \rho)$-subgaussian random variables $X$ ($\sigma>0, \rho>0$) as those
satisfying $\mathbb{E}(\exp(\lambda X))\leq \rho\exp(\frac{1}{2}\sigma^2\lambda^2)$
for any $\lambda\in\mathbb{R}$.
The introduction of the parameter $\rho$ may be particularly useful for those seeking to refine
bounds, or align results from different sources, in the analysis of stochastic processes and
concentration inequalities.
\end{abstract}

%\begin{keywords}
%  keyword one, keyword two, keyword three
%\end{keywords}

%%%%%%%%%%%%%%%%%%%%%%%%%%%%%%%%%%%%%%%%%%%%%%%%%%%%%%%%%%%%%%%%%%%%%%%%%%%%%%%%%%%%%%%%%%%%%%%%%%
%%%%%%%%%%%%%%%%%%%%%%%%%%%%%%%%%%%%%%%%%%%%%%%%%%%%%%%%%%%%%%%%%%%%%%%%%%%%%%%%%%%%%%%%%%%%%%%%%%
\section{Introduction\label{sec:intro}}
A subgaussian random variable refers to one whose probability distribution tail decays (at least)
as fast as that of a Gaussian variable. A well-known bound for the standard normal cumulative
distribution function $\Phi(x)$ is given by
\begin{equation}\label{eq:gaussian_tail}
\sqrt{\frac{2}{\pi}}\cdot\frac{1}{x+\sqrt{x^2+4}}\exp\left(-\frac{x^2}{2}\right)
< 1-\Phi(x) \leq
\sqrt{\frac{2}{\pi}}\cdot\frac{1}{x+\sqrt{x^2+8/\pi}}\exp\left(-\frac{x^2}{2}\right)
\end{equation}
for all $x\geq 0$~\cite{Mitrinovic1970}; accordingly, the tail probability of a subgaussian
variable is bounded above by $\rho \exp(-x^2/2\sigma^2)$ with some constants $\rho,\,\sigma^2>0$.
Actually, many distributions other than the Gaussian, such as the Gamma and Weibull distributions
with large shape parameters, and distributions with bounded supports, can exhibit subgaussian
properties. The concept of subgaussianity is useful in many applications in probability theory
and statistics, notably in analyzing random processes and deriving concentration inequalities.\par

Subgaussian variables can be defined through various equivalent characterizations, such as
upper-bounded tail probability, higher-order moments, moment-generating function, and Orlicz
norms~\cite{Vershynin2018,Wainwright2019,Philippe2015,Pauwels2020,Deligiannidis2021}.
Each characterization involves a ``variance proxy'' (such as the previously mentioned $\sigma^2$)
serving as a parameter to quantify statistical dispersion or the rate of tail decay, and the
variance proxies differ only by an absolute constant factor across characterizations. Textbooks
on subgaussianity commonly address this equivalence, though they typically do not focus on finding
potentially better bounds.
Existing texts often fix the values of $\rho$ or its counterparts at simple values like $1$ or $2$
while discussing the variance proxies; however, for those prioritizing lighter distribution tails,
it may be possible to com-promise the value of $\rho$ in pursuit of a more favorable $\sigma^2$.
On the other hand, different fixed values of $\rho$ may be chosen in applications, according to
the author's convention; for instance, both $\mathbb{E}\left[\exp(X^2/\sigma^2)\right]\leq 2$
\cite{Vershynin2018,Pauwels2020,Deligiannidis2021,Rivasplata2012} and
$\mathbb{E}\left[\exp(X^2/\sigma^2)\right]\leq\exp(1)$~\cite{Lan2012,Jin2019} are frequently used
assumptions, which can cause inconvenience when comparing the corresponding results.
It is also worth noting that some texts require subgaussian variables to be centered, depending on
the specific characterization used to define subgaussianity~\cite{Philippe2015,Rivasplata2012}.
For example, when a random variable $X$ is defined as subgaussian with its moment-generating
function satisfying $\mathbb{E}\left[\exp(\lambda X)\right]\leq \rho\exp(\sigma^2\lambda^2/2)$
($\lambda\in\mathbb{R}$), fixing $\rho$ at $1$ naturally implies
$\mathbb{E}\left[X\right]=0$~\cite{Rivasplata2012, Stromberg1994}.\par

This note aims to restate the basic properties of subgaussian variables by introducing the
parameter $\rho$ alongside the variance proxy. These properties include equivalent
characterizations of subgaussianity (without requiring the variables to be centered), its closure
under simple operations, and an application to martingale differences. While the derivations are
largely drawn from existing materials (with necessary modifications), I hope this flexible
treatment of subgaussianity will benefit those seeking refined bounds or aligned results from
different sources, in the analysis of large deviations and concentration inequalities.\par

%%%%%%%%%%%%%%%%%%%%%%%%%%%%%%%%%%%%%%%%%%%%%%%%%%%%%%%%%%%%%%%%%%%%%%%%%%%%%%%%%%%%%%%%%%%%%%%%%%
%%%%%%%%%%%%%%%%%%%%%%%%%%%%%%%%%%%%%%%%%%%%%%%%%%%%%%%%%%%%%%%%%%%%%%%%%%%%%%%%%%%%%%%%%%%%%%%%%%
\section{Characterizations and Properties\label{sec:properties}}

\subsection{Equivalent characterizations\label{sec:equivalence}}
Let us begin the main body with the adjusted definition of subgaussian variables, followed
by Theorem~\ref{Thm:equivalence}, which states the equivalence of various characterizations.\par

%%%%%%%%%%
\begin{definition}\label{Def:subgaussian}
A random variable $X\in\mathbb{R}$ is called $(\sigma, \rho)$-subgaussian if there exist constants
$\sigma>0$ and $\rho\geq 1$ such that the moment-generating function of $X$ satisfies
\begin{equation}\label{eq:equivalence_definition}
\mathbb{E}\left[\exp(\lambda X)\right]
\leq
 \rho\exp\Big(\frac{1}{2}\sigma^2 \lambda^2 \Big),
\text{\quad for all $\lambda\in\mathbb{R}$}.
\end{equation}
\end{definition}

%%%%%%%%%%
\begin{theorem}[Equivalent characterizations of subgaussianity\label{Thm:equivalence}]
For a random variable $X\in\mathbb{R}$ and constants $\sigma_1, \sigma_2, \sigma_3, \sigma_4,
\sigma_5>0$, $\rho_1, \rho_3, \rho_4\geq 1$, $\rho_2\geq 0$ and $\rho_5\geq 1/2$, the following
properties are equivalent in the sense that, there exist absolute constants $C_{ij}$ and mappings
$\varphi_{ij}$ such that the implication from $(\textup{j})$ to $(\textup{i})$ holds for all $\sigma_i$ and $\rho_i$ whenever $\sigma_i\geq C_{ij}\sigma_j$ and $\rho_i\geq\varphi_{ij}(\rho_j)$:
\begin{itemize}
\setlength\itemsep{0pt}
%%%%%%%%%%%%%%%%%%%%%
	\item[\textup{(1)}]
	The distribution tail of $\lvert X\rvert$ satisfies
	\begin{equation}\nonumber
		\mathbb{P}\left[ \lvert X\rvert \geq \lambda \sigma_1 \right]
		\leq \rho_1\exp(-\lambda^2/2), \text{\quad for all $\lambda\geq 0$}.
	\end{equation}
%%%%%%%%%%
	\item[\textup{(2)}]
	The moments of even orders of $X$ satisfy
	\begin{equation}\nonumber
		\mathbb{E}\,\big[ X^{2k} \big] \leq
		\rho_2 \cdot \sigma_2^{2k} \cdot k!, \text{\quad for all positive integers $k$}.
	\end{equation}
%%%%%%%%%%
	\item[\textup{(3)}]
	The moment-generating function of $X^2$ is finite at a specific point such that
	\begin{equation}\nonumber
		\mathbb{E}\left[ \exp( X^2/\sigma_3^2) \right]\leq \rho_3.
		\qquad\qquad%
	\end{equation}
%%%%%%%%%%
	\item[\textup{(4)}]
	$X$ is $(\sigma_4, \rho_4)$-subgaussian, namely, the moment-generating function of $X$ satisfies
	\begin{equation}\nonumber
		\mathbb{E}\left[ \exp(\lambda X) \right] \leq
		\rho_4 \exp\Big(\frac{1}{2}\sigma_4^2 \lambda^2\Big),
		\text{\quad for all $\lambda\in\mathbb{R}$}.
	\end{equation}
%%%%%%%%%%
	\item[\textup{(5)}]
	The distribution tails of $X$ satisfy
	\begin{equation}\nonumber
		\max\big\{ \mathbb{P}\left[ X \geq \lambda\sigma_5 \right],
		            \mathbb{P}\left[ X \leq -\lambda\sigma_5\right]\big\}
		\leq \rho_5\exp(-\lambda^2/2),
		\text{\quad for all $\lambda\geq 0$}.
	\end{equation}
\end{itemize}
\end{theorem}

~\pagebreak
\begin{remark}\label{Rmk:subgaussian1}
A set of applicable choices for $C_{ij}$ and $\varphi_{ij}(\rho)$ in Theorem~\ref{Thm:equivalence}
is listed in the table below, where $\lambda\in(0,1)$ is a parameter.
Note that these choices are not necessarily sharp.
\begin{table}[h]
%\caption{}
\begin{center}
\begin{threeparttable}
	\begin{tabular}{|c||*{5}{c|}}\hline
		\backslashbox{Concl.}{Hyp.}
		&\makebox[3em]{$(1)$}
		&\makebox[3em]{$(2)$}
		&\makebox[3em]{$(3)$}
		&\makebox[3em]{$(4)$}
		&\makebox[3em]{$(5)$}\\[6pt]\hline\hline
		%%%%%%%%%%%%%%%%%%
		$(1)$
		&	--
		&	$\sqrt{\frac{1}{2\lambda}},  1+\frac{\lambda\rho}{1-\lambda}$
		&	$\sqrt{\frac{1}{2}}, \rho$~\tnote{$\dagger$}
		&   $1, 2\rho$
		&	$1, 2\rho$~\tnote{$\dagger$}
		\\[5pt]
		%%%%%%%%%%%%%%%%%%
		$(2)$
		&	$\sqrt{2}, \rho$~\tnote{$\dagger$}
		&	--
		&	$1,       \rho-1$~\tnote{$\dagger$}
		&	$\sqrt{2}, 2\rho$
		&	$\sqrt{2}, 2\rho$
		\\[5pt]
		%%%%%%%%%%%%%%%%%%
		$(3)$
		&	$\sqrt{\frac{2}{\lambda}}, \frac{1}{1-\lambda}\rho^{\lambda}$~\tnote{$\dagger$}
		&	$\sqrt{\frac{1}{\lambda}},  1+ \frac{\lambda\rho}{1-\lambda}$~\tnote{$\dagger$}
		&	--
		&	$\sqrt{\frac{2}{\lambda}}$, Eq.~\eqref{eq:equivalence_varphi34}~\tnote{$\dagger$}
		&	$\sqrt{\frac{2}{\lambda}}, \frac{1}{1-\lambda}(2\rho)^\lambda$
		\\[5pt]
		%%%%%%%%%%%%%%%%%%
		$(4)$
		&	$\sqrt{\frac{1}{\lambda}}, \frac{1}{1-\lambda}\rho^\lambda$
		&	$\sqrt{\frac{1}{2\lambda}}, 1+\frac{\lambda\rho}{1-\lambda}$
		&	$\sqrt{\frac{1}{2}}, \rho$~\tnote{$\dagger$}
		&	--
		&	$\sqrt{\frac{1}{\lambda}}, \frac{1}{1-\lambda}(2\rho)^\lambda$
		\\[5pt]
		%%%%%%%%%%%%%%%%%%
		$(5)$
		&	$\sqrt{\frac{1}{\lambda}}, \frac{1}{1-\lambda}\rho^\lambda$
		&	$\sqrt{\frac{1}{2\lambda}}, 1+\frac{\lambda\rho}{1-\lambda}$
		&	$\sqrt{\frac{1}{2}}, \rho$
		&	$1, \rho$~\tnote{$\dagger$}
		&	--\\[5pt]\hline
	\end{tabular}
\end{threeparttable}
\begin{tablenotes}
\item[$\dagger$] {\small Derivations of the entries marked with $\dagger$ are provided in the proof.
The others simply follow from the transitivity of implications, and their details are omitted.}
%	\item $C_{41}=C_{43}C_{31}=\sqrt{\frac{1}{\lambda}}$,
%	$\varphi_{41}(\rho)=\varphi_{43}\circ\varphi_{31}(\rho)=\frac{1}{1-\lambda}\rho^\lambda$;\\
%	\item $C_{51}=C_{54}C_{41}=\sqrt{\frac{1}{\lambda}}$,
%	$\varphi_{51}(\rho)=\varphi_{54}\circ\varphi_{41}(\rho)=\frac{1}{1-\lambda}\rho^\lambda$;\\
%	\item $C_{12}=C_{13}C_{32}=\sqrt{1/2\lambda}$,
%	$\varphi_{12}(\rho)=\varphi_{13}\circ\varphi_{32}(\rho)=1+\frac{\lambda\rho}{1-\lambda}$;\\
%	\item $C_{42}=C_{43}C_{32}=\sqrt{\frac{1}{2\lambda}}$,
%	$\varphi_{42}(\rho)=\varphi_{43}\circ\varphi_{32}(\rho)=1+\frac{\lambda\rho}{1-\lambda}$;\\
%	\item $C_{52}=C_{54}C_{42}=\sqrt{\frac{1}{2\lambda}}$,
%	$\varphi_{52}(\rho)=\varphi_{54}\circ\varphi_{42}(\rho)=1+\frac{\lambda\rho}{1-\lambda}$;\\
%	\item $C_{53}=C_{54}C_{43}=\sqrt{1/2}$,
%	$\varphi_{53}(\rho)=\varphi_{54}\circ\varphi_{43}(\rho)=\rho$;\\
%	\item $C_{14}=C_{15}C_{54}=1$,
%	$\varphi_{14}(\rho)=\varphi_{15}\circ\varphi_{54}(\rho)=2\rho$;\\
%	\item $C_{24}=C_{21}C_{14}=\sqrt{2}$,
%	$\varphi_{24}(\rho)=\varphi_{21}\circ\varphi_{14}(\rho)=2\rho$;\\
%	\item $C_{25}=C_{21}C_{15}=\sqrt{2}$,
%	$\varphi_{25}(\rho)=\varphi_{21}\circ\varphi_{15}(\rho)=2\rho$;\\
%	\item $C_{35}=C_{31}C_{15}=\sqrt{\frac{2}{\lambda}}$,
%	$\varphi_{35}(\rho)=\varphi_{31}\circ\varphi_{15}(\rho)=\frac{1}{1-\lambda}(2\rho)^\lambda$;\\
%	\item $C_{45}=C_{41}C_{15}=\sqrt{\frac{1}{\lambda}}$,
%	$\varphi_{45}(\rho)=\varphi_{41}\circ\varphi_{15}(\rho)=\frac{1}{1-\lambda}(2\rho)^\lambda$;\\
\end{tablenotes}
\end{center}
\end{table}
\end{remark}

\begin{proof}
%%%%%%%%%%
%%%%%%%%%%
\\\indent
{\bf (1)$\implies$(2)}~(according to Ref.~\cite{Vershynin2018}): Given property~(1), we have
\begin{equation}\nonumber
	\begin{split}
		\mathbb{E}\left[ \lvert X\rvert^n \right]
		&= \int_0^\infty \mathbb{P}\left[ \lvert X\rvert^n \geq x^n \right] \dd x^n
		= n\int_0^\infty \mathbb{P}\left[ \lvert X\rvert \geq \lambda\sigma_1 \right]\,
		                 (\lambda\sigma_1)^{n-1} \dd(\lambda\sigma_1)\\
		&\leq n\rho_1\sigma_1^n
		\int_0^\infty \exp(-\lambda^2/2)\,\lambda^{n-1} \dd\lambda
		= \frac{1}{2} n\rho_1 \big(\sqrt{2}\sigma_1\big)^n
		\int_0^\infty \exp(-x)\,x^{\tfrac{n-2}{2}} \dd x\\
		&= \rho_1\cdot \big(\sqrt{2}\sigma_1\big)^n \cdot\Gamma\left(\frac{n}{2}+1\right)
	\end{split}
\end{equation}
for any $n\geq 1$, where $\Gamma(\,\cdot\,)$ denotes the gamma function.
When $n=2k$ with $k=1,2,3\dots$, property~(2) is obtained with $\sigma_2\geq\sqrt{2}\sigma_1$
and $\rho_2\geq \rho_1$; that is, $C_{21}=\sqrt{2}$ and $\varphi_{21}(\rho) = \rho$.\par
%%%%%%%%%%
\vspace{3pt}
%%%%%%%%%%
{\bf (2)$\implies$(3)}~(according to Ref.~\cite{Vershynin2018}, with adjustments):
Given property~(2), we find that the moment-generating function of $X^2$ satisfies
\begin{equation}\nonumber
\begin{split}
	\mathbb{E}\left[ \exp(\lambda X^2/\sigma_2^2 ) \right]
	&=    \sum_{k=0}^\infty \frac{\lambda^k\, \mathbb{E}[ X^{2k}]}{k!\,\sigma_2^{2k}}\,\\[-3pt]
	&\leq 1 + \rho_2\sum_{k=1}^\infty \lambda^k
	=     1 + \frac{\lambda \rho_2}{1-\lambda},
\end{split}
\end{equation}
for all $\lambda\in [0,1)$.
We observe that the rightmost side of the above expression increases with $\lambda$;
we can freely select any $\lambda\in(0,1)$ and accordingly set
$C_{32}=\sqrt{1/\lambda}$, $\varphi_{32}(\rho)=1+\lambda\rho/(1-\lambda)$,
with which property~(3) holds.\par
%%%%%%%%%%
\vspace{3pt}
%%%%%%%%%%
{\bf (3)$\implies$(4)}~(according to \cite{Wainwright2019}, with adjustments):
From property~(3) and the inequality $\lambda X\leq \lambda^2\sigma_3^2/4 + X^2/\sigma_3^2$,
we find
\begin{equation}\nonumber
\begin{split}
	\mathbb{E}\left[ \exp(\lambda X) \right]
	&\leq \mathbb{E}\left[\exp(\lambda^2\sigma_3^2/4+ X^2/\sigma_3^2) \right]\\
	&=    \exp(\lambda^2\sigma_3^2/4)\cdot \mathbb{E}\left[ \exp(X^2/\sigma_3^2)\right]\\
	&\leq \rho_3 \exp\big( \lambda^2 \sigma_3^2/4 \big)
\end{split}
\end{equation}
for all $\lambda\in\mathbb{R}$.
Therefore, property~(4) holds for $\sigma_4\geq\sqrt{2}\sigma_3/2$
and $\rho_4\geq\rho_3$, i.e., $C_{43}=\sqrt{1/2}$ and $\varphi_{43}(\rho)=\rho$.\par
%%%%%%%%%%
\vspace{3pt}
%%%%%%%%%%
{\bf (4)$\implies$(5)}~(generic Chernoff bound, see~\cite{Vershynin2018, Wainwright2019}):
Given property~(4), Markov's inequality implies
\begin{equation}\nonumber
\begin{split}
	\mathbb{P}\left[ X \geq \lambda\sigma_5 \right]
	&=    \mathbb{P}\left[ \exp( x X) \geq \exp(x\lambda\sigma_5) \right]\\
	&\leq \mathbb{E}\left[ \exp(x X) \right]\exp(-x\lambda\sigma_5)\\
	&\leq \rho_4\exp\left(\sigma_4^2 x^2/2 - \lambda\sigma_5 x\right)
\end{split}
\end{equation}
for all $\lambda\geq 0$ and $x>0$.
Note that
$\mathbb{P}\left[X\geq\lambda\sigma_5\right]\leq\rho_4\exp(\sigma_4^2 x^2/2-\lambda\sigma_5 x)$
holds trivially for $x=0$. Minimizing the right-hand side with respect to $x$, we obtain
\begin{equation}\nonumber
\begin{split}
\mathbb{P}\left[ X \geq \lambda \sigma_5 \right]
	&\leq \rho_4\inf_{x\geq 0} \exp\left(\sigma_4^2 x^2/2 - \lambda\sigma_5 x\right)\\
	&= \rho_4\exp\left[-{\lambda^2 \sigma_5^2}/2\sigma_4^2\right],
\end{split}
\end{equation}
where the infimum is attained at $x = \lambda\sigma_5/\sigma_4^2$.
Since property~(4) is invariant when replacing $X$ with $-X$, we similarly find
$\mathbb{P}[ -X \geq \lambda \sigma_5 ] \leq \rho_4\exp[-{\lambda^2 \sigma_5^2}/2\sigma_4^2]$
for all $\lambda\geq 0$.
Therefore, property~(5) holds whenever $\sigma_5\geq \sigma_4$ and $\rho_5\geq \rho_4$, i.e.,
$C_{54}=1$ and $\varphi_{54}(\rho)=\rho$.\par
%%%%%%%%%%
\vspace{3pt}
%%%%%%%%%%
{\bf (5)$\implies$(1)}:
This is trivial, with $C_{15}=1$ and $\varphi_{15}(\rho)=2\rho$.\par

So far, we have demonstrated the equivalence of all characterizations; any $C_{ij}$ and
$\varphi_{ij}$ not directly addressed can be determined using the transitivity of implications.
Furthermore, improvements can be made to certain $\varphi_{ij}$ with the following additional
implications.\par

%%%%%%%%%%
\vspace{3pt}
%%%%%%%%%%
{\bf (3)$\implies$(1)}~\cite{Vershynin2018}:
This follows immediately from Markov's inequality, since
\begin{equation}\nonumber
\begin{split}
\mathbb{P}\left[ \lvert X\rvert \geq \lambda\sigma_3/\sqrt{2} \right]
	&=    \mathbb{P}\left[ X^2/\sigma_3^2 \geq \lambda^2/2 \right]\\
	&\leq \mathbb{E}\left[ \exp(X^2/\sigma_3^2) \right] \exp(-\lambda^2/2)\\
	&\leq \rho_3\exp(-\lambda^2/2)
\end{split}
\end{equation}
for any $\lambda\geq 0$, given property~(3).
Therefore, property~(1) holds for all $\sigma_1\geq \sqrt{2}\sigma_3/2$ and $\rho_1\geq \rho_3$,
i.e., $C_{13}=\sqrt{1/2}$ and $\varphi_{13}(\rho)=\rho$. The value $C_{13}=\sqrt{1/2}$ obtained here
equals that derived via chain of implications, $C_{15}C_{54}C_{43}$, while $\varphi_{15}(\rho)=\rho$
obtained here is slightly better than the composition
$\varphi_{15}\circ\varphi_{54}\circ\varphi_{43}(\rho)=2\rho$.\par
%%%%%%%%%%
\vspace{3pt}
%%%%%%%%%%
{\bf (3)$\implies$(2)}~\cite{Deligiannidis2021}:
By expanding property~(3) into a power series, we obtain
\begin{equation}\nonumber
	\mathbb{E}\left[ \exp( X^2/\sigma_3^2) \right]
	=1 + \sum_{k=1}^\infty \frac{\mathbb{E}\left[X^{2k}\right]}{k!\,\sigma_3^{2k}}
	\leq \rho_3.
\end{equation}
Since every term in the sum is non-negative, it follows that
\begin{equation}\nonumber
	\mathbb{E}\,\big[X^{2k}\big] \leq (\rho_3-1)\cdot \sigma_3^{2k}\cdot k!
\end{equation}
for all positive integers $k$. Therefore, property~(2) holds, with $C_{23}=1$ and
$\varphi_{23}(\rho)=\rho-1$, where $\varphi_{23}(\rho)$ is slightly better than the composition
$\varphi_{21}\circ\varphi_{13}(\rho)=\rho$.\par
%%%%%%%%%%
\vspace{3pt}
%%%%%%%%%%
{\bf (1)$\implies$(3)}:
This can be derived from the chain (1)$\implies$(2)$\implies$(3) or directly as shown in
Refs.~\cite{Deligiannidis2021,Rivasplata2012}. However, it should be noted that property~(1) may
lead to a slightly better conclusion than property~(2). Given property~(1), we have
\begin{equation}\nonumber
\begin{split}
\mathbb{E}\left[ \lvert X\rvert^n \right]
	&=    n\int_0^\infty \mathbb{P}\left[ \lvert X\rvert \geq \lambda\sigma_1 \right]\,
	      (\lambda\sigma_1)^{n-1} \dd (\lambda\sigma_1)\\
	&\leq n\sigma_1^n \int_0^\infty
	      \min\big\{\rho_1 \exp(-\lambda^2/2), 1\big\}\,\lambda^{n-1} \dd\lambda\\
	&=    \frac{1}{2} n\big(\sqrt{2}\sigma_1\big)^n \int_0^\infty
	      \min\big\{ \rho_1\exp(-x), 1\big\}\, x^{\tfrac{n-2}{2}} \dd x\\
	&=    \frac{1}{2} n\big(\sqrt{2}\sigma_1\big)^n
		  \bigg(\int_0^{\ln \rho_1}\mkern-5mu x^{\tfrac{n-2}{2}} \dd x +
	      \rho_1\int_{\ln \rho_1}^\infty\mkern-5mu\exp(-x)\,x^{\tfrac{n-2}{2}} \dd x\bigg)\\
	&=    \rho_1\cdot \big(\sqrt{2}\sigma_1\big)^n \cdot
		  \Gamma\left(\frac{n}{2}+1,\, \ln\rho_1 \right)
\end{split}
\end{equation}
for any $n\geq 1$, where $\Gamma(\cdot\,,\cdot)$ denotes the upper incomplete gamma function.
When $n=2k$ with $k=1,2,3\dots$, we have
\begin{equation}\nonumber
\mathbb{E}\big[ X^{2k} \big]
	\leq \rho_1\cdot 2^k\sigma_1^{2k} \cdot \Gamma\left(k+1,\, \ln\rho_1\right)
	=    \sum_{i=0}^{k}\frac{(\ln\rho_1)^i}{i !} \cdot 2^k\sigma_1^{2k} \cdot k!, 
\end{equation}
which is slightly better than the result
$\mathbb{E}\big[ X^{2k} \big]\leq \rho_1 \cdot 2^k\sigma_1^{2k} \cdot k!$ in the derivation
for implication~(1)$\implies$(2). Furthermore, we find
\begin{equation}\nonumber
\begin{split}
	\mathbb{E}\left[ \exp(\lambda X^2/2\sigma_1^2 ) \right]
	&=    \sum_{k=0}^\infty \frac{\lambda^k\, \mathbb{E}[ X^{2k}]}{k!\,2^k\sigma_1^{2k}}
	\leq  \sum_{k=0}^\infty \lambda^k \sum_{i=0}^k\frac{(\ln\rho_1)^i}{i!} \\[-3pt]
	&=    \sum_{i=0}^\infty \frac{(\ln\rho_1)^i}{i!} \sum_{k=i}^\infty \lambda^k
	=     \frac{1}{1-\lambda}\sum_{i=0}^\infty \frac{(\lambda\ln\rho_1)^i}{i!}
	=     \frac{ \rho_1^\lambda}{1-\lambda},
\end{split}
\end{equation}
for all $\lambda\in[0,1)$.
Therefore, we can freely select any $\lambda\in(0,1)$, and property~(3) holds with
$C_{31}=\sqrt{2/\lambda}$ and $\varphi_{31}(\rho)=\rho^\lambda/(1-\lambda)$,
where $\varphi_{31}(\rho)$ is slightly better than the composition
$\varphi_{32}\circ\varphi_{21}(\rho)=1+\lambda\rho/(1-\lambda)$.\par
%%%%%%%%%%
\vspace{3pt}
%%%%%%%%%%
{\bf (4)$\implies$(3)} (according to~\cite{Wainwright2019}):
Starting from property~(4), we have
\begin{equation}\nonumber
\mathbb{E}\bigg[ \exp\Big(x X - \frac{1}{2\lambda}\sigma_4^2\,x^2 \Big)\bigg] \leq
\rho_4\exp\bigg( \frac{1}{2}\Big(1-\frac{1}{\lambda}\Big)\sigma_4^2 x^2 \bigg)
\end{equation}
for all $x\in\mathbb{R}$ and $\lambda\in (0, 1)$.
By integrating both sides with respect to $x$ on $(-\infty,+\infty)$ and using Fubini's
theorem, we obtain
\begin{equation}\nonumber
\frac{\sqrt{2\lambda\pi}}{\sigma_4}
	\mathbb{E}\left[ \exp\Big( \frac{\lambda X^2}{2\sigma_4^2} \Big) \right]\leq 
	\frac{\rho_4}{\sigma_4 }\cdot\frac{\sqrt{2\lambda\pi}}{\sqrt{1-\lambda}},
\end{equation}
which simplifies to
\begin{equation}\nonumber
	\mathbb{E}\left[ \exp\Big( \frac{\lambda X^2}{2\sigma_4^2} \Big) \right]\leq
	\frac{\rho_4}{\sqrt{1-\lambda}}
\end{equation}
for all $\lambda\in (0, 1)$.
Therefore, property~(3) holds with $C_{34}=\sqrt{2/\lambda}$ and
$\varphi_{34}(\rho)=\rho/\sqrt{1-\lambda}$.
Meanwhile, the implication chain (4)$\implies$(5)$\implies$(1)$\implies$(3) leads to $C_{31}C_{15}C_{54}=\sqrt{2/\lambda}$
and $\varphi_{31}\circ\varphi_{15}\circ\varphi_{54}=(2\rho)^\lambda/(1-\lambda)$.
We may finally choose
\begin{equation}\label{eq:equivalence_varphi34}
	\varphi_{34}(\rho) = (1-\lambda)^{-1} 
	\min\left\{ \rho\sqrt{1-\lambda},\, (2\rho)^\lambda\right\}.
\end{equation}
\end{proof}

\begin{remark}\label{Rmk:subgaussian2}
If $X$ is either non-negative or non-positive, i.e., $\mathbb{P}\left[X\geq 0\right]=1$ or
$\mathbb{P}\left[X\leq 0\right]=1$, then properties~\textup{(1)} and~\textup{(5)} in
Theorem~\ref{Thm:equivalence} will be the same.
The corresponding table for $C_{ij}$ and $\varphi_{ij}(\rho)$
in Remark~\ref{Rmk:subgaussian1} is shown below, where $\lambda\in(0,1)$ is a parameter.
\begin{table}[h]
%\caption{}
\begin{center}
	\begin{threeparttable}
		\begin{tabular}{|c||*{4}{c|}}\hline
			\backslashbox{Concl.}{Hyp.}
			&\makebox[3em]{$(1)$}
			&\makebox[4em]{$(2)$}
			&\makebox[3em]{$(3)$}
			&\makebox[3em]{$(4)$}\\[6pt]\hline\hline
			%%%%%%%%%%%%%%%%%%
			$(1)$
			&	--
			&	~$\sqrt{\frac{1}{2\lambda}},  1+\frac{\lambda\rho}{1-\lambda}$~
			&	~$\sqrt{\frac{1}{2}}, \rho$~\tnote{$\dagger$}~
			&	$1, \rho$~\tnote{$\dagger$}
			\\[5pt]
			%%%%%%%%%%%%%%%%%%
			$(2)$
			&	$\sqrt{2}, \rho$~\tnote{$\dagger$}
			&	--
			&	~$1,       \rho-1$~\tnote{$\dagger$}~
			&	$\sqrt{2}, \rho$
			\\[5pt]
			%%%%%%%%%%%%%%%%%%
			$(3)$
			&	$\sqrt{\frac{2}{\lambda}}, \frac{1}{1-\lambda}\rho^{\lambda}$~\tnote{$\dagger$}
			&	~$\sqrt{\frac{1}{\lambda}},  1+ \frac{\lambda\rho}{1-\lambda}$~\tnote{$\dagger$}~
			&	--
			&	$\sqrt{\frac{2}{\lambda}}$,
				$\min\left\{ \frac{1}{\sqrt{1-\lambda}}\rho, \,
				\frac{1}{1-\lambda}\rho^{\lambda}\right\}$~\tnote{$\dagger$}
			\\[5pt]
			%%%%%%%%%%%%%%%%%%
			$(4)$
			&	$\sqrt{\frac{1}{\lambda}}, \frac{1}{1-\lambda}\rho^\lambda$
			&	~$\sqrt{\frac{1}{2\lambda}}, 1+\frac{\lambda\rho}{1-\lambda}$~
			&	~$\sqrt{\frac{1}{2}}, \rho$~\tnote{$\dagger$}~
			&	--
			\\[5pt]\hline
		\end{tabular}
	\end{threeparttable}
\begin{tablenotes}
	\item[$\dagger$] {\small Entries derived with a proof.
		The others follow from the transitivity of implications.}
\end{tablenotes}
\end{center}
\end{table}
\end{remark}

\subsection{Centered subgaussian variables\label{sec:centered}}
%%%%%%%%%%%%%%%%%%%%%%%%%%%%%%%%%%%%%%%%%%%%%%%%%%%%%%%%%%%%%%%%%%%%%%%%%%%%%%%%%%%%%%%%%%%%%%%%%%
Definition~\ref{Def:subgaussian} and the equivalent characterizations in
Theorem~\ref{Thm:equivalence} do not require subgaussian variables to be centered (by
``subgaussian'' we mean that any of the properties in Theorem~\ref{Thm:equivalence} is satisfied).
Only $(\sigma, 1)$-subgaussian variables, as conventionally defined, are guaranteed to have a zero
mean. Conversely, any centered subgaussian variable must be $(\sigma, 1)$-subgaussian for some
$\sigma$.\par

%%%%%%%%%%
\begin{theorem}[Centered subgaussian variables\label{Thm:centered}]
Let $X\in\mathbb{R}$ be a subgaussian variable, satisfying any of the properties in
Theorem~\ref{Thm:equivalence}, and assume that $\mathbb{E}\left[X\right]=0$.
Then $X$ must be $(\sigma, 1)$-subgaussian, namely, its moment-generating function satisfies
\begin{equation}\nonumber
\mathbb{E}\left[ \exp(\lambda X) \right] \leq \exp\Big(\frac{1}{2}\sigma^2 \lambda^2\Big),
\end{equation}
for all $\lambda\in\mathbb{R}$, where $\sigma$ can be determined as follows:
\begin{itemize}
\setlength\itemsep{0pt}
	%%%%%%%%%%%%%%%%%%%%%
	\item[\textup{(a)}]
	If property~$(3)$ in Theorem~\ref{Thm:equivalence} is satisfied, then
	\begin{equation}\nonumber
		\frac{\sigma^2}{\sigma_3^2} =
		\left\{
		\begin{aligned}
			& \frac{1}{2} + \frac{9}{8}\ln\rho_3,
			&&\textup{for } \ln\rho_3 \in \big[0, 4/9\big);\\
			& \frac{3}{2}\sqrt{\ln\rho_3},
			&&\textup{for } \ln\rho_3 \in \big[4/9, 16/9\big);\\
			& \frac{9}{8}\ln\rho_3,
			&&\textup{for } \ln\rho_3 \in \big[16/9, +\infty\big).
		\end{aligned}\right.
	\end{equation}
	%%%%%%%%%%%%%%%%%%%%%
	\item[\textup{(b)}]	
	If property~$(2)$ in Theorem~\ref{Thm:equivalence} is satisfied, then
	\begin{equation}\nonumber
		\frac{\sigma^2}{\sigma_2^2} =
		\left\{
		\begin{aligned}
			& \frac{29}{45},
			&&\textup{for } \rho_2 \in \big[0, 29/90\big);\\
			& \frac{1}{2}\bigg[
			\sqrt{\Big(\rho_2 - \frac{19}{90}\Big)^2 + \frac{26}{15}\rho_2 }
			+ \bigg(\rho_2 + \frac{19}{90}\bigg)\bigg],
			&&\textup{for } \rho_2 \in \big[29/90, +\infty\big).
		\end{aligned}\right.
	\end{equation}
	%%%%%%%%%%%%%%%%%%%%%
	\item[\textup{(c)}]	
	If property~$(1)$ in Theorem~\ref{Thm:equivalence} is satisfied, then
	\begin{equation}\nonumber
		\frac{\sigma^2}{\sigma_1^2} = \frac{1}{3}\big(8+7\ln\rho_1\big).
	\end{equation}
	%%%%%%%%%%%%%%%%%%%%%
	\item[\textup{(d)}]
	If property~$(4)$ in Theorem~\ref{Thm:equivalence} is satisfied, then
	\begin{equation}\nonumber
		\frac{\sigma^2}{\sigma_4^2}
		=\frac{9}{8}\cdot
		\frac{\ln\rho_4 + \sqrt{(\ln\rho_4)^2 + 2\ln\rho_4}}{\ln\rho_4}
		\left[ \ln\rho_4 +
		\frac{1}{2}\ln\left( 1 + \ln\rho_4 + \sqrt{(\ln\rho_4)^2 + 2\ln\rho_4} \right)
		\right].
	\end{equation}
\end{itemize}
\end{theorem}

\begin{proof}
%%%%%%%%%%
%%%%%%%%%%
\\\indent
{\bf Case (a)}~(adjusted from the proof for Lemma~2 in Ref.~\cite{Lan2012}):
From the numerical relation $\exp(x)\leq x + \exp(9x^2/16)$ ($x\in\mathbb{R}$), we find
\begin{equation}
\begin{split}\label{eq:centered_proof_case_a1}
\mathbb{E}\left[ \exp(\lambda X) \right]
	&\leq \lambda\mathbb{E}\left[X\right] +
	 \mathbb{E}\left[ \exp(\tfrac{9}{16}\lambda^2 X^2 ) \right]\\
	&\leq 0 +
	\mathbb{E}\left[\exp(X^2/\sigma_3^2)\right]^{{\tfrac{9}{16}\lambda^2 \sigma_3^2}}\\
	&\leq
	\exp\left( \tfrac{9}{16}\lambda^2\sigma_3^2\cdot \ln\rho_3 \right)
\end{split}
\end{equation}
for any $\lambda$ such that $9\lambda^2\sigma_3^2/16\leq 1$, i.e.,
$\lvert\lambda\rvert\sigma_3\leq 4/3$, where we used the assumption $\mathbb{E}\left[X\right]=0$,
Jensen's inequality, and property~(3) in Theorem~\ref{Thm:equivalence}. On the other hand, from the
inequality $\lambda X \leq x\lambda^2\sigma_3^2/4 + x^{-1}X^2/\sigma_3^2$ for $x>0$, we find
\begin{equation}\nonumber
\begin{split}
\mathbb{E}\left[ \exp(\lambda X) \right]
	&\leq \mathbb{E}\left[\exp( x\lambda^2\sigma_3^2/4 + x^{-1} X^2/\sigma_3^2) \right]\\
	&\leq \exp\left( x\lambda^2\sigma_3^2/4 \right)\cdot
	 \mathbb{E}\left[ \exp( X^2/\sigma_3^2) \right]^{1/x}\\
	&\leq \exp\left( x\lambda^2\sigma_3^2/4 + x^{-1}\ln\rho_3 \right)\\[-2pt]
\end{split}
\end{equation}
for all $x\geq 1$, where we again used Jensen's inequality and property~(3). By minimizing the
rightmost side with respect to $x$, we obtain
\begin{equation}\label{eq:centered_proof_case_a2}
\mathbb{E}\left[ \exp(\lambda X) \right]\leq
\left\{
\begin{aligned}
	& \exp\Big(  \sqrt{\ln\rho_3}\,\sigma_3\lvert\lambda\rvert \Big),
	&&\textup{for }\lvert\lambda\rvert\sigma_3 \leq 2\sqrt{\ln\rho_3};\\
	& \exp\Big( \tfrac{1}{4}\sigma_3^2 \lambda^2 + \ln\rho_3 \Big),
	&&\textup{for }\lvert\lambda\rvert\sigma_3 > 2\sqrt{\ln\rho_3}.
\end{aligned}
\right.
\end{equation}
Combining inequalities~\eqref{eq:centered_proof_case_a1} and~\eqref{eq:centered_proof_case_a2}, the
upper bound of the moment-generating function of $X$ on $\lambda\in R$ can be expressed as follows:
When $2\sqrt{\ln \rho_3}\leq 4/3$ (i.e., $\ln\rho_1\leq 4/9$), we have
\begin{equation}\nonumber
\mathbb{E}\left[ \exp(\lambda X)\right]
\leq\left\{
\begin{aligned}
	& \exp\Big( \tfrac{9}{16}\ln\rho_3 \cdot \sigma_3^2 \lambda^2 \Big),
	&&\textup{for }\lvert\lambda\rvert\sigma_3 \leq 4/3;\\
	& \exp\Big[ \big(\tfrac{1}{4} + \tfrac{9}{16}\ln\rho_3\big)\sigma_3^2\,\lambda^2 \Big],
	&&\textup{for }\lvert\lambda\rvert\sigma_3 > 4/3.
\end{aligned}
\right.
\end{equation}
When $\ln\rho_1> 4/9$, we have
\begin{equation}\nonumber
\mathbb{E}\left[ \exp(\lambda X) \right]\leq
\left\{
\begin{aligned}
	& \exp\Big( \tfrac{9}{16}\ln\rho_3\cdot\sigma_3^2  \lambda^2\Big),
	&&\textup{for }\lvert\lambda\rvert\sigma_3 \in \big[0, 4/3\big];\\
	& \exp\Big( \tfrac{3}{4}\sqrt{\ln\rho_3}\cdot\sigma_3^2 \lambda^2\Big),
	&&\textup{for }\lvert\lambda\rvert\sigma_3 \in \big(4/3, 2\sqrt{\ln\rho_3}\,\big];\\
	& \exp\Big(  \tfrac{1}{2} \sigma_3^2 \lambda^2 \Big),
	&&\textup{for }\lvert\lambda\rvert\sigma_3 \in \big(2\sqrt{\ln\rho_3}, +\infty\big).
\end{aligned}
\right.
\end{equation}
Note that
\begin{equation}\nonumber
\max\Big\{ \tfrac{9}{16}\ln\rho_3,\, \tfrac{3}{4}\sqrt{\ln\rho_3},\, \tfrac{1}{2} \Big\}\\
=\left\{
\begin{aligned}
	& \tfrac{9}{16}\ln\rho_3,		&&\textup{for } \ln\rho_1 > 16/9;\\
	& \tfrac{3}{4}\sqrt{\ln\rho_3},	&&\textup{for } 4/9 \leq\ln\rho_1 \leq 16/9.\\
\end{aligned}
\right.
\end{equation}
By choosing the greatest value of the coefficient for $\sigma_3^2\lambda^2$ in each interval
$\ln\rho_3\in [0, 4/9)$, $\ln\rho_3\in [4/9, 16/9)$, or $\ln\rho_3\in [16/9, +\infty)$, the claim
of case~(a) is proved.\par
%%%%%%%%%%
\vspace{3pt}
%%%%%%%%%%
{\bf Case (b)} (adjusted from~\cite{Wainwright2019}):
The moments of odd orders of a random variable $Y$ can be bounded according to the Cauchy--Schwarz
inequality, as
\begin{equation}\nonumber
\mathbb{E}\left[ Y^{2k+1} \right]
\leq
	\sqrt{\mathbb{E}\left[ Y^{2k}\right]\,\mathbb{E}\left[ Y^{2k+2}\right]}
\leq \frac{1}{2}\left(
	 x \mathbb{E}\big[ Y^{2k}\big] + x^{-1}\mathbb{E}\big[ Y^{2k+2}\big]\right)
\end{equation}
for all $x>0$ and $k=0,1,2,\dots$. Applying this for $\lambda X$ and substituting it into the power
series expansion of $\mathbb{E}\left[ \exp(\lambda X)\right]$, we find
\begin{equation}\nonumber
\begin{split}
\mathbb{E}\left[ \exp(\lambda X) \right]
	&= 1 + \lambda\mathbb{E}\left[X\right] + 
	\sum_{k=1}^\infty\left\{
	\frac{\mathbb{E}\left[(\lambda X)^{2k}  \right]}{(2k)!} +
	\frac{\mathbb{E}\left[(\lambda X)^{2k+1}\right]}{(2k+1)!}\right\} \\
	%%%%%
	&\leq 1 + \sum_{k=1}^\infty\left\{
	\frac{\lambda^{2k}\mathbb{E}[X^{2k}]}{(2k)!} +
	\frac{x_k     \lambda^{2k}\mathbb{E}[X^{2k}]+
		  x_k^{-1}\lambda^{2k+2}\mathbb{E}[X^{2k+2}]}{2\cdot(2k+1)!}
	\right\}\\[-2pt]
	%%%%%
	&\leq 1 + \rho_2\sum_{k=1}^\infty\left\{
	\frac{        \lambda^{2k}  \sigma_2^{2k}\,k!      }{(2k)!} +
	\frac{x_k     \lambda^{2k}  \sigma_2^{2k}\,k! +
		  x_k^{-1}\lambda^{2k+2}\sigma_2^{2k+2}\,(k+1)!}
	  	 {2\cdot(2k+1)!}
	\right\}\\[-2pt]
	&=   1 + \rho_2\cdot\frac{\lambda^2\sigma_2^2}{2!}\left(1+\frac{x_1}{6}\right)
	       + \rho_2\sum_{k=2}^\infty
		\frac{\lambda^{2k}\sigma_2^{2k}k!}{(2k)!}
		\left(\frac{k}{x_{k-1}}+1+\frac{x_k}{4k+2}\right)\\[-2pt]
	&=   1 + \sum_{k=1}^\infty c_k\frac{\lambda^{2k}\sigma_2^{2k}}{k!},
\end{split}
\end{equation}
where we have used the assumption $\mathbb{E}\left[X\right]=0$ and property~(2), and have defined
coefficients
\begin{equation}\nonumber
c_k\coloneqq\left\{
\begin{aligned}
	&  \frac{1}{2}\left(1+\frac{x_1}{6}\right)\rho_2,
	~&&\textup{for }k=1;\\
	&  \frac{(k!)^2}{(2k)!}\left(\frac{k}{x_{k-1}}+1+\frac{x_k}{4k+2}\right)\rho_2,
	~&&\textup{for }k=2,3,\dots,
\end{aligned}
\right.
\end{equation}
where $x_k>0$ are to be determined. By choosing $x_k=2(k+1)$ for $k\geq2$, we have
\begin{equation}\nonumber
c_k=\left\{
\begin{aligned}
	&  \frac{1}{2}\left(1+\frac{x_1}{6}\right)\rho_2,
	~&&\textup{for }k=1;\\
	&  \frac{1}{6}\left(\frac{2}{x_1}+\frac{8}{5}\right)\rho_2,
	~&&\textup{for }k=2;\\
	&  \frac{(k!)^2}{(2k)!}\left( 2 + \frac{1}{4k+2}\right)\rho_2,
	~&&\textup{for }k=3,4,\dots,
\end{aligned}
\right.
\end{equation}
and
\begin{equation}\nonumber
\left\{\begin{aligned}
c_1	&=	  \frac{1}{2}\left(1+\frac{x_1}{6}\right)\rho_2;\\
\frac{c_2}{c_1}	&=
	\frac{1}{3}\cdot\frac{2/x_1 + 8/5}{1+x_1/6} =
	\frac{2}{3x_1}\frac{1+\frac{4}{5}x_1}{1+\frac{1}{6}x_1} \leq
	\frac{2}{3x_1}\left[\frac{19}{120}\Big(x_1-6\Big) +\frac{29}{10} \right],
	~&&\textup{when }x_1\in(0, 6];\\
\frac{c_3}{c_2} &=
	\frac{3}{10}\cdot \frac{ 2 + 1/14 }{ 2/x_1 + 8/5 }\leq \frac{9}{28}<\frac{29}{90},
	~&&\textup{when }x_1\in(0, 6];\\
\frac{c_{k+1}}{c_k} &=
	\frac{(k+1)}{4(k+3/2)}\cdot\frac{(8k+13)}{(8k+5)}<0.29,
	~&&\textup{for }k=3,4,\dots.
\end{aligned}
\right.
\end{equation}
Now we need to choose the value of $x_1$ appropriately and find a common upper bound for $c_1$,
$c_2/c_1$, and all $c_{k+1}/c_k$ ($k\geq 2$). When $\rho_2\leq {29}/{90}$, we can choose $x_1=6$,
so that $c_1\leq 29/90$ and $c_{k+1}/c_k\leq 29/90$ for all $k\geq 1$.
When $\rho_2> {29}/{90}$, we choose $x_1=x_1^\ast\in(0,6)$ such that
\begin{equation}
	\frac{1}{2}\left(1+\frac{x_1^\ast}{6}\right)\rho_2
	= \frac{2}{3x_1^\ast}\left[\frac{19}{120}\Big(x_1^\ast-6\Big) +\frac{29}{10} \right]
	> \frac{29}{90},
\end{equation}
which leads to an explicit solution
\begin{equation}\nonumber
x_1^\ast = \frac{3}{\rho_2}\bigg[
\sqrt{\Big(\rho_2 - \frac{19}{90}\Big)^2 + \frac{26}{15}\rho_2 }
-\left(\rho_2 - \frac{19}{90}\right)\bigg]
\end{equation}
and accordingly
\begin{equation}\nonumber
c_1^\ast = \frac{1}{4}\bigg[
\sqrt{\Big(\rho_2 - \frac{19}{90}\Big)^2 + \frac{26}{15}\rho_2 }
+ \left(\rho_2 + \frac{19}{90}\right)\bigg].
\end{equation}
With such choices of $x_k$ ($k=1,2,3,\dots$), we finally have
\begin{equation}\nonumber
\begin{split}
\mathbb{E}\left[ \exp(\lambda X) \right]
&\leq  1 + \sum_{k=1}^\infty c_k\frac{\lambda^{2k}\sigma_2^{2k}}{k!}
 \leq  1 + \sum_{k=1}^\infty \frac{\lambda^{2k}\sigma_2^{2k}}{k!}
	\max\left\{ \frac{29}{90}, c_1^\ast\right\}^k\\
&\leq
	\exp\left( \max\Big\{\frac{29}{90}, c_1^\ast\Big\}\sigma_2^2\lambda^2\right),
\end{split}
\end{equation}
which proves the claim of case~(b).\par
%%%%%%%%%%
\vspace{3pt}
%%%%%%%%%%
{\bf Case (c)}: Property~(1) in Theorem~\ref{Thm:equivalence} indicates
$\mathbb{E}\,[ X^{2k}]\leq  f_k(\rho_1)\cdot(\sqrt{2}\sigma_1)^{2k}\cdot k!$, where
$f_k(\rho_1)\coloneqq\sum_{i=0}^k (\ln\rho_1)^i/i!$ for $k=1,2,3,\dots$, as shown in the proof for
(1)$\implies$(3) therein.
Similar to the proof above for case~(b), it is easy to find
\begin{equation}\nonumber
\begin{split}
&\mathbb{E}\left[ \exp(\lambda X) \right]
	\leq 1 + \sum_{k=1}^\infty\left\{
	\frac{\lambda^{2k}\mathbb{E}[X^{2k}]}{(2k)!} +
	\frac{x_k     \lambda^{2k}\mathbb{E}[X^{2k}]+
		x_k^{-1}\lambda^{2k+2}\mathbb{E}[X^{2k+2}]}{2\cdot(2k+1)!}
	\right\}\\[-2pt]
	%%%%%
	\leq\,& 1 + \sum_{k=1}^\infty\left\{
	\frac{        \big(\sqrt{2}\lambda\sigma_1\big)^{2k} f_k(\rho_1)\,k!      }{(2k)!} +
	\frac{x_k     \big(\sqrt{2}\lambda\sigma_1\big)^{2k} f_k(\rho_1)\,k! +
		  x_k^{-1}\big(\sqrt{2}\lambda\sigma_1\big)^{2k+2}f_{k+1}(\rho_1)\,(k+1)!}
	{2\cdot(2k+1)!}
	\right\}\\[-2pt]
	=\,&   1 +
	\frac{ \big(\sqrt{2}\lambda\sigma_1\big)^2}{2!}\left(1+\frac{x_1}{6}\right)f_1(\rho_1)
	+ \sum_{k=2}^\infty
	\frac{ \big(\sqrt{2}\lambda\sigma_1\big)^{2k}k! }{(2k)!}
	\left(\frac{k}{x_{k-1}}+1+\frac{x_k}{4k+2}\right)f_k(\rho_1)\\[-2pt]
	=\,& 1 + \sum_{k=1}^\infty c_k(\rho_1)\frac{\big(2\lambda^2\sigma_1^2\big)^k }{k!}
\end{split}
\end{equation}
for all $\lambda\in\mathbb{R}$, where, by choosing $x_k=2(k+1)$ for $k\geq2$, the coefficients $c_k$
are given by
\begin{equation}\nonumber
	c_k(\rho_1)=\left\{
	\begin{aligned}
		&  \frac{1}{2}\left(1+\frac{x_1}{6}\right) f_1(\rho_1),
		~&&\textup{for }k=1;\\
		&  \frac{1}{6}\left(\frac{2}{x_1}+\frac{8}{5}\right) f_2(\rho_1),
		~&&\textup{for }k=2;\\
		&  \frac{(k!)^2}{(2k)!}\left( 2 + \frac{1}{4k+2}\right) f_k(\rho_1),
		~&&\textup{for }k=3,4,\dots.
	\end{aligned}
	\right.
\end{equation}
Following from the inequality
\begin{equation}\nonumber
\Big(1+\frac{x}{k+1}\Big)\sum_{i=0}^k
\frac{x^i}{i!}\geq \sum_{i=0}^{k+1}\frac{x^i}{i!}
\end{equation}
for $x\geq 0$ and $k\geq 0$, we have
\begin{equation}\nonumber
\left\{\begin{aligned}
	c_1	&=	  \frac{1}{2}\left(1+\frac{x_1}{6}\right)\left(1+\ln\rho_1\right);\\
	\frac{c_2}{c_1}	&=
	\frac{1}{3}\cdot\frac{2/x_1 + 8/5}{1+x_1/6} \cdot\frac{f_2(\rho_1)}{f_1(\rho_1)}\leq
	\frac{2}{3x_1}\frac{1+\frac{5}{6}x_1}{1+\frac{1}{6}x_1}\cdot\frac{f_3(\rho_1)}{f_2(\rho_1)}\\
	&\leq \frac{2}{3x_1}\left[\frac{3}{8}\Big(x_1-2\Big) +2 \right]
	\left(1+\frac{1}{2}\ln\rho_1\right),
	~&&\textup{when }x_1\in(0, 2];\\
	\frac{c_3}{c_2} &=
	\frac{3}{10}\cdot \frac{ 2 + 1/14 }{ 2/x_1 + 8/5 }\cdot\frac{f_3(\rho_1)}{f_2(\rho_1)}
	\leq \frac{87}{364}\left( 1+\frac{1}{3}\ln\rho_1\right),
	~&&\textup{when }x_1\in(0, 2];\\
	\frac{c_{k+1}}{c_k} &=
	\frac{(k+1)}{4(k+3/2)}\cdot\frac{(8k+13)}{(8k+5)}\cdot\frac{f_{k+1}(\rho_1)}{f_k(\rho_1)}
	<0.29\left( 1+\frac{1}{4}\ln\rho_1\right),
	~&&\textup{for }k=3,4,\dots.
\end{aligned}
\right.
\end{equation}
By choosing $x_1=x_1^\ast\in(0,2]$ such that
\begin{equation}
\frac{1}{2}\left(1+\frac{x_1^\ast}{6}\right)\left(1+\ln\rho_1\right)
=
\frac{2}{3x_1^\ast}\left[\frac{3}{8}\Big(x_1^\ast-2\Big) +2 \right]\left(1+\frac{1}{2}\ln\rho_1\right)
\geq \frac{2}{3},
\end{equation}
which leads to an explicit solution
\begin{equation}\nonumber
x_1^\ast = \frac{\sqrt{196 + 348\ln\rho_1 + 161(\ln\rho_1)^2}-(6 + 9\ln\rho_1)}{4(1 + \ln\rho_1)},
\end{equation}
we have $c_1(\rho_1)\leq c_1^\ast(\rho_1)$ and $c_{k+1}(\rho_1)/c_k(\rho_1)\leq c_1^\ast(\rho_1)$
for all $k\geq 1$ and all $\rho_1\geq 1$, where
\begin{equation}\nonumber
c_1^\ast(\rho_1) = 
\frac{\sqrt{196 + 348\ln\rho_1 + 161(\ln\rho_1)^2} +18 + 15\ln\rho_1}{48}
\leq
\frac{7\ln\rho_1+8}{12}.
\end{equation}
Therefore, we finally have
\begin{equation}\nonumber
\begin{split}
	\mathbb{E}\left[ \exp(\lambda X) \right]
	&\leq  1 + \sum_{k=1}^\infty c_k(\rho_1)\frac{ \big(2\lambda^2\sigma_1^2\big)^k }{k!}
	 \leq  1 + \sum_{k=1}^\infty c_1^\ast(\rho_1)^k
	 			\frac{ \big(2\lambda^2\sigma_1^2\big)^k }{k!}\\
	&\leq \exp\left( 2c_1^\ast(\rho_1)\sigma_1^2\lambda^2\right)
	 \leq \exp\left( \frac{7\ln\rho_1+8}{6} \sigma_1^2\lambda^2\right),
\end{split}
\end{equation}
which proves the claim of case~(c).\par
%%%%%%%%%%
\vspace{3pt}
%%%%%%%%%%
{\bf Case (d)}:
According to Theorem~\ref{Thm:equivalence} and Remark~\ref{Rmk:subgaussian1}, if property~(4) is
satisfied, then property~(3) is satisfied, with
$\sigma_3=\sigma_4\sqrt{2/x}$ and $\rho_3 = \rho_4/\sqrt{1-x}$ for any $x\in(0, 1)$.
Furthermore, given $\mathbb{E}\left[X\right]=0$, the proof for case~(a) indicates
\begin{equation}\nonumber
\mathbb{E}\left[ \exp(\lambda X) \right]
\leq
\exp\left( \frac{9}{16}\lambda^2\sigma_3^2\cdot \ln\rho_3 \right)
=
\exp\left( \lambda^2\sigma_4^2\cdot \frac{9}{8x}\ln\frac{\rho_4}{\sqrt{1-x}}\right)
\end{equation}
for all $\lambda$ such that $\lvert\lambda\rvert\sigma_3\leq 4/3$, i.e.,
$\lvert\lambda\rvert\sigma_4\leq \frac{4}{3}\sqrt{x/2}$. On the other hand, for all $\lambda$ such
that $\lvert\lambda\rvert\sigma_4\geq \frac{4}{3}\sqrt{x/2}$, we clearly have
\begin{equation}\nonumber
\mathbb{E}\left[ \exp(\lambda X) \right] \leq
\exp\Big(\frac{1}{2}\sigma_4^2 \lambda^2 + \ln\rho_4 \Big) \leq
\exp\Big(\frac{1}{2}\sigma_4^2 \lambda^2 + \frac{9\sigma_4^2\lambda^2}{8x}\ln\rho_4 \Big) =
\exp\Big(\sigma_4^2 \lambda^2\cdot\frac{4x+9\ln\rho_4}{8x} \Big).
\end{equation}
Also, we have $9(\ln\rho_4 - \ln\sqrt{1-x})\geq 9\ln\rho_4 + 4x$ for $x\in(0, 1)$,
since $-9\ln(1-x)\geq 8x$ can be verified easily.
Therefore, for all $\lambda\in\mathbb{R}$ we have
\begin{equation}\nonumber
\mathbb{E}\left[ \exp(\lambda X) \right] \leq
\exp\left( \frac{9}{8}\lambda^2\sigma_4^2
\min_{x\in(0,1)}\frac{1}{x}\ln\frac{\rho_4}{\sqrt{1-x}} \right),
\end{equation}
where the minimizing $x$ tends to $0$ and $1$ when $\rho_4$ approaches $1$ and $+\infty$,
respectively. Since the minimizing $x$ does not have an explicit expression, we choose the
following surrogate
\begin{equation}
x^\ast = \sqrt{(\ln\rho_4)^2 + 2\ln\rho_4} -\ln\rho_4
  = \frac{2\ln\rho_4}{\sqrt{(\ln\rho_4)^2 + 2\ln\rho_4}+\ln\rho_4},
\end{equation}
which leads to $(1-x^\ast)^{-1} = 1 + \ln\rho_4 + \sqrt{(\ln\rho_4)^2 + 2\ln\rho_4}$, and thus
\begin{equation}\nonumber
\begin{split}
&\min_{x\in(0,1)}\frac{1}{x}\ln\frac{\rho_4}{\sqrt{1-x}}
\leq
\frac{1}{x^\ast}\ln\frac{\rho_4}{\sqrt{1-x^\ast}}\\
%%%%%%%%%%
&\quad=
\frac{\sqrt{(\ln\rho_4)^2 + 2\ln\rho_4}+\ln\rho_4}{2\ln\rho_4}
\left[ \ln\rho_4 +
\frac{1}{2}\ln\left( 1 + \ln\rho_4 + \sqrt{(\ln\rho_4)^2 + 2\ln\rho_4}  \right)
\right].
\end{split}
\end{equation}
Therefore, the claim in case (d) is proved.
\end{proof}

Note that the $\sigma^2$ values for $(\sigma, 1)$-subgaussian variables, explicitly provided in
Theorem~\ref{Thm:centered}, are derived directly from each of the subgaussian properties and are
not meant to be optimal. By translating between different equivalent properties, one can
potentially find a better $\sigma$, as demonstrated in the following example.\par

\begin{example}
Assume that $X$ is centered and satisfies
$\mathbb{E}\left[ X^{2k} \right] \leq \rho_2\cdot \sigma_2^{2k}\cdot k!$ for $k=1,2,3,\dots$, where
$\rho_2$ takes on the values $5\times10^{-3}$, $1$, or $10$.
According to case~\textup{(b)} of Theorem~\ref{Thm:centered}, $X$ is
$(0.803\,\sigma_2, 1)$-, $(1.18\,\sigma_2, 1)$-, or $(3.23\,\sigma_2, 1)$-subgaussian for each
respective value of $\rho_2$.\par
%%%%%%%%%%
Meanwhile, Theorem~\ref{Thm:equivalence} and Remark~\ref{Rmk:subgaussian1} imply that
$\mathbb{E}\left[ \exp( 9X^2/10\sigma_2^2) \right]\leq 1+9\rho_2$. Consequently,
case~\textup{(a)} of Theorem~\ref{Thm:centered} indicates that $X$ is
$(0.782\,\sigma_2, 1)$-, $(1.70\,\sigma_2, 1)$-, or $(2.38\,\sigma_2, 1)$-subgaussian,
respectively.
\end{example}

While some texts define subgaussian variables by Eq.~\eqref{eq:equivalence_definition} fixing
$\rho=1$, some others~\cite{Wainwright2019,Deligiannidis2021} generalize the concept by allowing
nonzero expectations. They consider $X$ subgaussian if its centered version,
$X-\mathbb{E}\left[X\right]$, satisfies Definition~\ref{Def:subgaussian} with $\rho=1$;
this implies that a subgaussian variable plus any constant remains subgaussian. This treatment is
actually "equivalent" to the definition of $(\sigma, \rho)$-subgaussian variables provided here, as
will be clarified in Theorem~\ref{Thm:constant} and Corollary~\ref{Cor:centered}.\par

\begin{theorem}\label{Thm:constant}
Let $X\in\mathbb{R}$ be a $(\sigma, \rho)$-subgaussian random variable
and $c\in\mathbb{R}$ be a constant.
Then $X+c$ is $(\sigma', \rho')$-subgaussian, with
\begin{equation}\label{eq:centered_constant}
\frac{\sigma'}{\sigma} = \sqrt{1+x},\quad
\frac{\rho'}{\rho}     = \exp\Big(\frac{c^2}{2x\sigma^2} \Big),
\end{equation}
where $x$ is an arbitrary positive number.
\end{theorem}

\begin{corollary}
\label{Cor:centered}
A random variable $X\in\mathbb{R}$ is $(\sigma, \rho)$-subgaussian for some constants $\sigma>0$
and $\rho\geq 1$, if and only if $X-\mathbb{E}\left[X\right]$ is $(\sigma', 1)$-subgaussian for
some $\sigma'>0$.
\end{corollary}

\begin{proof}
Given the assumptions in Theorem~\ref{Thm:constant}, we have
\begin{equation}\nonumber
\begin{split}
\mathbb{E}\left[\exp\big(\lambda (X+c) \big)\right]
&=   
\exp\left(c\lambda\right)\mathbb{E}\left[\exp\left(\lambda X \right)\right]\\
&\leq
\rho\exp\Big(\frac{1}{2}\sigma^2\lambda^2 + c\lambda\Big)\\
&\leq
\rho\exp\left(\frac{1}{2}\sigma^2\lambda^2 + c\lambda
        + \frac{1}{2}x\sigma^2\Big( \lambda -\frac{c}{x\sigma^2} \Big)^2\right)\\
&=
\rho\exp\left(\frac{1}{2}\big(1+x\big)\sigma^2\lambda^2 
        + \frac{c^2}{2x\sigma^2}\right)
\end{split}
\end{equation}
for all $\lambda\in\mathbb{R}$ and $x>0$, hence proving Theorem~\ref{Thm:constant}.
Applying Theorem~\ref{Thm:constant} with $c=\pm\,\mathbb{E}\left[X\right]$
and Theorem~\ref{Thm:centered} leads to the corollary.
\end{proof}

\subsection{Closure under simple operations\label{sec:closure}}
Theorem~\ref{Thm:constant} may be viewed as a specific instance of the closure of subgaussianity
under sum{-\linebreak}mation, as a constant $c\in\mathbb{R}$ is trivially a
$\big(\sigma, \exp(c^2/2\sigma^2)\big)$-subgaussian variable,
according to Definition~\ref{Def:subgaussian}.
Some more general cases of the closure of subgaussianity are demonstrated by
Theorems~\ref{Thm:closure1} and~\ref{Thm:closure2}, where the discussion is based on properties~(4)
and~(3) in Theorem~\ref{Thm:equivalence}, respectively.
Obviously, the closure of subgaussianity can also be
expressed with respect to the other subgaussian properties, potentially by introducing additional
absolute constants.\par

%%%%%%%%%%%%%%%%%%%%%%%%%%%%%%%%%%%%%%%%%%%%%%%%%%%%%%%%%%%%%%%%%%%%%%%%%%%%%%%%%%%%%%%%%%%%%%%%%%
\begin{theorem}\label{Thm:closure1}
Let $X_i\in\mathbb{R}$ $(i=1,2,3,\dots,n)$ be $n$ random variables that are
$(\sigma_i, \rho_i)$-subgaussian, respectively. Then we have
\begin{itemize}
\item[\textup{(i)}]
$X\coloneqq \sum_{i=1}^n X_i$ is $(\sigma, \rho)$-subgaussian, with
\begin{equation}\nonumber
	\sigma = \sum_{i=1}^n\sigma_i,\quad
	\ln\rho = \frac{\sum_{i=1}^n\sigma_i\ln\rho_i}{\sum_{i=1}^n\sigma_i};
\end{equation}
%%%%%%%%%%
\item[\textup{(ii)}] $X\coloneqq \sum_{i=1}^n X_i$,
where all $X_i$ are independent from one another, is $(\sigma, \rho)$-subgaussian, with
\begin{equation}\nonumber
	\sigma = \sqrt{\textstyle\sum_{i=1}^n\sigma_i^2},\quad
	\rho = \prod_{i=1}^n \rho_i;
\end{equation}
%%%%%%%%%%
\item[\textup{(iii)}] $X\coloneqq \max_{i=1,2,\dots,n}\{X_i\}$ is $(\sigma, \rho)$-subgaussian, with
\begin{equation}\nonumber
	\sigma =  \max_{1\leq i\leq n}\{\sigma_i\},\quad
	\rho   =  \sum_{i=1}^n \rho_i.
\end{equation}
\end{itemize}
\end{theorem}

\begin{proof}
%%%%%%%%%%
%%%%%%%%%%
\\\indent
{\bf Case (i)}~(adjusted from Ref.~\cite{Rivasplata2012}):
Considering the case of $n=2$, we have
\begin{equation}\nonumber
\begin{split}
&\mathbb{E}\left[\exp(\lambda X\right)]
=\mathbb{E}\left[\exp(\lambda X_1)\,\exp(\lambda X_2)\right]\\
&\leq
\mathbb{E}\left[\exp\Big(\tfrac{\sigma_1+\sigma_2}{\sigma_1}
	\lambda X_1\Big)\right]^{\frac{\sigma_1}{\sigma_1+\sigma_2}}
\mathbb{E}\left[\exp\Big(\tfrac{\sigma_1+\sigma_2}{\sigma_2}
	\lambda X_2\Big)\right]^{\frac{\sigma_2}{\sigma_1+\sigma_2}}\\
&\leq
\left[\rho_1 \exp\left( \frac{1}{2}\sigma_1^2\cdot
	\frac{\lambda^2 (\sigma_1+\sigma_2)^2}{\sigma_1^2}\right)
	\right]^{\frac{\sigma_1}{\sigma_1+\sigma_2}}
\left[\rho_2 \exp\left( \frac{1}{2}\sigma_2^2\cdot
	\frac{\lambda^2 (\sigma_1+\sigma_2)^2}{\sigma_2^2}\right)
	\right]^{\frac{\sigma_2}{\sigma_1+\sigma_2}}\\
&=
\exp\left( \frac{1}{2}(\sigma_1+\sigma_2)^2\lambda^2
+ \frac{\sigma_1\ln\rho_1 + \sigma_2\ln\rho_2}{\sigma_1+\sigma_2}
 \right)
\end{split}
\end{equation}
for all $\lambda\in\mathbb{R}$, where we applied H\"{o}lder's inequality for the first inequality.
The conclusion generalizes to larger $n$ by induction.\par
%%%%%%%%%%
\vspace{3pt}
%%%%%%%%%%
{\bf Case (ii)}:
Given the assumption, we have
\begin{equation}\nonumber
\begin{split}
\mathbb{E}\left[\exp(\lambda X\right)]
&=\mathbb{E}\Big[ {\textstyle\prod_{i=1}^n} \exp(\lambda X_i) \Big]
%%%%%
={\textstyle\prod_{i=1}^n} \mathbb{E}\left[\exp(\lambda X_i)\right]\\
&\leq
{\displaystyle\prod_{i=1}^n}\;
	\rho_i \exp\Big( \frac{1}{2}\sigma_i^2\lambda^2\Big)
= \Big({\textstyle\prod_{i=1}^n} \rho_i\Big)\cdot
\exp\bigg( \frac{1}{2}\lambda^2\sum_{i=1}^n \sigma_i^2\bigg)
\end{split}
\end{equation}
for all $\lambda\in\mathbb{R}$, where we used the independence of all $X_i$.\par
%%%%%%%%%%
\vspace{3pt}
%%%%%%%%%%
{\bf Case (iii)}:
Given the assumption, we have
\begin{equation}\nonumber
\begin{split}
\mathbb{E}\left[ \exp(\lambda X) \right] =
\mathbb{E}\Big[ \exp\big(\lambda\max_{1\leq i\leq n}\{X_i\}\big) \Big]
&<
\mathbb{E}\Big[ {\textstyle \sum_{i=1}^n} \exp(\lambda X_i) \Big]\\
\leq
{\textstyle \sum_{i=1}^n} \rho_i\exp\big(\tfrac{1}{2}\sigma_i^2\lambda^2\big)
&\leq
\Big( {\textstyle \sum_{i=1}^n} \rho_i \Big)
\exp\Big[ \tfrac{1}{2}\lambda^2 \max_{1\leq i\leq n}\{\sigma_i^2\} \Big]
\end{split}
\end{equation}
for any $\lambda\in\mathbb{R}$, hence proving the claim.
\end{proof}

\begin{theorem}\label{Thm:closure2}
Let $X_i\in\mathbb{R}$ $(i=1,2,3,\dots,n)$ be $n$ random variables satisfying
$\mathbb{E}\left[ \exp(X_i^2/\sigma_i^2) \right]\leq \rho_i$, respectively.
Then we have
\begin{itemize}
\item[\textup{(i)}]
	$X$ such that $\lvert X\rvert \leq \sum_{i=1}^n \lvert X_i\rvert$ satisfies
	$\mathbb{E}\left[\exp(X^2/\sigma^2)\right]\leq \rho$, with
	\begin{equation}\nonumber
		\sigma = \sum_{i=1}^n\sigma_i,\quad
		\ln\rho = \frac{\sum_{i=1}^n\sigma_i\ln\rho_i}{\sum_{i=1}^n\sigma_i};
	\end{equation}
%%%%%%%%%%
\item[\textup{(ii)}]
	$X$ such that $\lvert X\rvert\leq \sqrt{\sum_{i=1}^n X_i^2}$ satisfies
	$\mathbb{E}\left[\exp(X^2/\sigma^2)\right]\leq \rho$, with 
	\begin{equation}\nonumber
		\sigma = \sqrt{\textstyle\sum_{i=1}^n\sigma_i^2},\quad
		\ln\rho = \frac{\sum_{i=1}^n\sigma_i^2\ln\rho_i}{\sum_{i=1}^n\sigma_i^2};
	\end{equation}
%%%%%%%%%%
\item[\textup{(iii)}]
	$X\coloneqq \sum_{i=1}^n X_i$, where all $X_i$ are independent from one another, satisfies
	$\mathbb{E}\left[\exp(X^2/\sigma^2)\right]\leq \rho$, with
	\begin{equation}\nonumber
		\sigma = \sqrt{\textstyle\sum_{i=1}^n\sigma_i^2},\quad
		\rho = \prod_{i=1}^n \rho_i;
	\end{equation}
\end{itemize}
\end{theorem}

\begin{proof}
\\\indent
{\bf Case~(i)}:
Considering the case of $n=2$, we have
\begin{equation}\nonumber
\frac{X^2}{(\sigma_1+\sigma_2)^2}
\leq
\frac{(\lvert X_1\rvert+\lvert X_2\rvert)^2}{(\sigma_1+\sigma_2)^2}
\leq
\frac{(1+\sigma_2\sigma_1^{-1})X_1^2+(1+\sigma_1\sigma_2^{-1})X_2^2}{(\sigma_1+\sigma_2)^2}
=
\frac{\sigma_1^{-1}X_1^2 + \sigma_2^{-1}X_2^2}{\sigma_1+\sigma_2},
\end{equation}
and we further find
\begin{equation}\nonumber
\begin{split}
\mathbb{E}\left[\exp\Big(\tfrac{X^2}{(\sigma_1+\sigma_2)^2}\Big)\right]
&\leq
\mathbb{E}\left[
\exp\Big(\tfrac{\sigma_1}{\sigma_1+\sigma_2}\cdot {X_1^2}/{\sigma_1^2}\Big)\cdot
\exp\Big(\tfrac{\sigma_2}{\sigma_1+\sigma_2}\cdot {X_2^2}/{\sigma_2^2}\Big)\right] \\
&\leq
\mathbb{E}\Big[\exp\big({X_1^2}/{\sigma_1^2}\big)\Big]^{\frac{\sigma_1}{\sigma_1+\sigma_2}}
\cdot
\mathbb{E}\Big[\exp\big({X_2^2}/{\sigma_2^2}\big)\Big]^{\frac{\sigma_2}{\sigma_1+\sigma_2}}\\
&\leq
\rho_1^{{\sigma_1}/(\sigma_1+\sigma_2)}\rho_2^{{\sigma_2}/(\sigma_1+\sigma_2)},
\end{split}
\end{equation}
according to H\"{o}lder's inequality. The conclusion generalizes to larger $n$ by induction.\par
%%%%%%%%%%
\vspace{3pt}
%%%%%%%%%%
{\bf Case~(ii)}:
Considering the case of $n=2$, we have
\begin{equation}\nonumber
\begin{split}
\mathbb{E}\left[\exp\Big(\tfrac{X^2}{\sigma_1^2+\sigma_2^2}\Big)\right]
&\leq
\mathbb{E}\left[
\exp\Big(\tfrac{\sigma_1^2}{\sigma_1^2+\sigma_2^2}\cdot {X_1^2}/{\sigma_1^2}\Big)\cdot
\exp\Big(\tfrac{\sigma_2^2}{\sigma_1^2+\sigma_2^2}\cdot {X_2^2}/{\sigma_2^2}\Big)
\right] \\
&\leq
\mathbb{E}\Big[\exp\big({X_1^2}/{\sigma_1^2}\big)\Big]^{\sigma_1^2/(\sigma_1^2+\sigma_2^2)}
\cdot
\mathbb{E}\Big[\exp\big({X_2^2}/{\sigma_2^2}\big)\Big]^{\sigma_2^2/(\sigma_1^2+\sigma_2^2)}\\
&\leq
\rho_1^{{\sigma_1^2}/(\sigma_1^2+\sigma_2^2)}
\rho_2^{{\sigma_2^2}/(\sigma_1^2+\sigma_2^2)},
\end{split}
\end{equation}
according to H\"{o}lder's inequality. The conclusion generalizes to larger $n$ by induction.\par
%%%%%%%%%%
\vspace{3pt}
%%%%%%%%%%
{\bf Case~(iii)}:
Given the assumption, we have
\begin{equation}\nonumber
\begin{split}
\mathbb{E}\left[
\exp\Big( X^2/{\textstyle \sum_{i=1}^n}\sigma_i^2  \Big)\right]
&=
\mathbb{E}\left[
\exp\Big( \big({\textstyle \sum_{i=1}^n}X_i\big)^2/{\textstyle \sum_{i=1}^n}\sigma_i^2
\Big)\right]
\leq
\mathbb{E}\left[
\exp\Big( {\textstyle \sum_{i=1}^n} \big(X_i^2/\sigma_i^2\big)\Big)\right]\\
&=
\textstyle{\prod_{i=1}^n \mathbb{E}\Big[ \exp\big( X_i^2/\sigma_i^2 \big)\Big]
\leq
\prod_{i=1}^n \rho_i},
\end{split}
\end{equation}
where we used Sedrakyan's inequality for the first sign of inequality and the independence of all
$X_i$ for the second.
\end{proof}

\subsection{Martingale difference with subgaussianity\label{sec:martingale}}
%%%%%%%%%%%%%%%%%%%%%%%%%%%%%%%%%%%%%%%%%%%%%%%%%%%%%%%%%%%%%%%%%%%%%%%%%%%%%%%%%%%%%%%%%%%%%%%%%%
A martingale is a sequence of random variables where the expected values remain unchanged over time,
given its past history. Martingales are widely used in the study of stochastic processes, including
fair gambling, asset price changes, algorithms for stochastic optimization, and more.
In this note, we consider vector-valued martingales with subgaussian differences and apply the
results from previous subsections to conduct a large deviation analysis of the martingales, as
summarized in the following Theorem~\ref{Thm:martingale}. This theorem, where assumptions (I), (II),
and (III) progressively loosen the subgaussianity condition of the vector martingale differences,
compiles existing results from~\cite[Lemma~2]{Lan2012},~\cite[Lemma~6]{Jin2019},
\cite[Theorem 2.2.2]{Pauwels2020}, and~\cite[Theorem 7]{Deligiannidis2021}.\par

%%%%%%%%%%%%%%%%%%%
\begin{theorem}\label{Thm:martingale}
%%%%%%%%%%%%%%%%%%%
Let $\{X_i\}_{i=1,2,3\dots}$ be a stochastic process, $\{\mathcal{F}_i\}_{i=1,2,3\dots}$ be the
filtrations of corresponding $\sigma$-fields up to time $i$,
and let $\bm{\phi}_i = [\phi_{i1}~\phi_{i2}\;\dots\;\phi_{id} ]^\top \in \mathbb{R}^d$ be given by
deterministic measurable functions $\bm{\phi}_i = \bm{\phi}_i(X_1, X_2, \dots, X_i)$ such that
$\mathbb{E}\left[\bm{\phi}_i \vert \mathcal{F}_{i-1} \right]= \bm{0}$ for all $i$.
Furthermore, we consider the following conditions:
\begin{itemize}
%%%%%%%%%
\item[\textup{(I)}]
	$\mathbb{E}\big[
	\exp(\phi_{ij}^2/\sigma_{ij}^2)\big\vert \mathcal{F}_{i-1}\big] \leq \exp(1)$
	where $\sum_{j=1}^d \sigma_{ij}^2\leq \sigma_i^2$,
	for all $i=1,2,3,\dots,n$ and $j=1,2,3,\dots,d$,
	with $\sigma_{ij}, \sigma_i>0$;
\vspace{1pt}
%%%%%%%%%
\item[\textup{(II)}]
	$\mathbb{E}\big[
	\exp(\norm{\bm{\phi}_i}^2\mkern-2.5mu/\sigma_i^2)\big\vert \mathcal{F}_{i-1}\big]
	\leq \exp(1)$ for all $i=1,2,3,\dots, n$, with $\sigma_i>0$;
\vspace{1pt}
%%%%%%%%%
\item[\textup{(III)}]
	$\mathbb{E}\big[ \exp\big((\bm{e}_\textup{u}^\top\bm{\phi}_i)^2/\sigma_i^2\big)
	\big\vert \mathcal{F}_{i-1}\big] \leq \exp(1)$ for any unit vector
	 $\bm{e}_\textup{u}\in\mathbb{R}^d$ and all $i=1,2,3\dots,n$,
	 with $\sigma_i>0$.
\end{itemize}
Then for any $\lambda\geq 0$ we have
\begin{equation}
\begin{split}
&\mathbb{P}\bigg[ \Big\lVert{\textstyle\sum_{i=1}^n} \bm{\phi}_i\Big\rVert
	\geq
	\lambda\sqrt{{\textstyle\sum_{i=1}^n}\sigma_i^2} \bigg]
%%%%%
\leq \left\{
\begin{aligned}
	&2\exp(-\lambda^2/4),                &&\textup{if (I) holds;}\\[2pt] 
	&(d+1)\exp(-\lambda^2/3),            &&\textup{if (II) holds;}\\[2pt]
	&5^d \exp\left(-\lambda^2/12\right), &&\textup{if (III) holds,}
\end{aligned}
\right.
\end{split}
\end{equation}
and have
\begin{equation}
\mathbb{P}\bigg[
\bm{e}_\textup{u}^\top{\textstyle\sum^n_{i=1}\bm{\phi_i}}
\geq \lambda\sqrt{{\textstyle\sum^n_{i=1}}\sigma_i^2} \bigg]
\leq
\exp(-\lambda^2/3)
\end{equation}
for any unit vector $\bm{e}_\textup{u}\in\mathbb{R}^d$
given any of the conditions~\textup{(I)}, \textup{(II)}, and~\textup{(III)}.
\end{theorem}

\begin{proof}
%%%%%%%%%%
\vspace{3pt}
%%%%%%%%%%
\\\indent
{\bf Case (I)}:
Given that $\mathbb{E}\left[\bm{\phi}_i \vert \mathcal{F}_{i-1} \right] = \bm{0}$ and
$ \mathbb{E}\big[\exp(\phi_{ij}^2/\sigma_{ij}^2) \vert \mathcal{F}_{i-1}\big]\leq\exp(1)$,
Theorem~\ref{Thm:centered} (case~(a)) indicates
\begin{equation}\nonumber
\mathbb{E}\big[ \exp( x\phi_{ij}) \big\vert \mathcal{F}_{i-1} \big]\leq
	\exp\big( \tfrac{3}{4}\sigma_{ij}^2 x^2 \big)
\end{equation}
for any $x\in\mathbb{R}$, $1\leq i\leq n$, and $1\leq j\leq d$.
Then we have
\begin{equation}\label{eq:martingale_total_expectation}
\begin{split}
\mathbb{E}\big[ \exp(x {\textstyle\sum_{i=1}^n}\phi_{ij}) \big]
&=
\mathbb{E}\Big[ \exp(x {\textstyle\sum_{i=1}^{n-1}}\phi_{ij})
	\cdot\mathbb{E}\big[\exp(x \phi_{nj}) \big\vert \mathcal{F}_{n-1}\big]
	\Big]\\
&\leq
\exp\big( \tfrac{3}{4}\sigma_{nj}^2 x^2 \big)\cdot
\mathbb{E}\big[ \exp( x{\textstyle\sum_{i=1}^{n-1}}\phi_{ij}) \big]\\
&\leq \cdots \leq 
\exp\Big[ \tfrac{3}{4} x^2 {\textstyle\sum_{i=1}^n}\sigma_{ij}^2 \Big]
\end{split}
\end{equation}
for any $x\in\mathbb{R}$ and $1\leq j\leq d$. Following from the implication (4)$\implies$(3) in
Theorem~\ref{Thm:equivalence} and Remark~\ref{Rmk:subgaussian1}, we find
\begin{equation}\nonumber
\mathbb{E}\bigg[
	\exp\bigg( \frac{x(\sum_{i=1}^n\phi_{ij})^2}{3\sum_{i=1}^n\sigma_{ij}^2}
	\bigg)\bigg]
\leq
	\frac{1}{\sqrt{1-x}}
\end{equation}
for any $x\in(0, 1)$.
As $\norm{\sum_{i=1}^n \bm{\phi}_i}^2 = \sum_{j=1}^d (\sum_{i=1}^n \phi_{ij})^2$,
invoking assumption (I) and applying Theorem~\ref{Thm:closure2} (case (ii)), we find
\begin{equation}\nonumber
\mathbb{E}\bigg[
\exp\bigg( \frac{x \norm{\sum_{i=1}^n \bm{\phi}_i}^2}{3\sum_{i=1}^n\sigma_i^2}
\bigg)\bigg]
\leq
\mathbb{E}\bigg[
\exp\bigg( \frac{x \norm{\sum_{i=1}^n \bm{\phi}_i}^2}{3\sum_{j=1}^d\sum_{i=1}^n\sigma_{ij}^2}
\bigg)\bigg]
\leq
\frac{1}{\sqrt{1-x}}
\end{equation}
for any $x\in(0, 1)$. Finally, according to the implication (3)$\implies$(1) in
Theorem~\ref{Thm:equivalence}, we find
\begin{equation}\nonumber
\mathbb{P}\bigg[ \Big\lVert{\textstyle\sum_{i=1}^n} \bm{\phi}_i\Big\rVert
\geq \lambda \sqrt{\tfrac{3}{2x} {\textstyle\sum_{i=1}^n}\sigma_i^2  } \bigg]
\leq
\frac{1}{\sqrt{1-x}} \exp(-\lambda^2/2),
\end{equation}
or equivalently
\begin{equation}
\mathbb{P}\bigg[ \Big\lVert{\textstyle\sum_{i=1}^n} \bm{\phi}_i\Big\rVert
\geq
	\lambda\sqrt{{\textstyle\sum_{i=1}^n}\sigma_i^2  } \bigg]
\leq
	\frac{1}{\sqrt{1-x}} \exp(-x \lambda^2/3),
\end{equation}
for any $x\in(0, 1)$ and $\lambda\geq 0$. Choosing $x=3/4$ leads to
\begin{equation}
\mathbb{P}\bigg[ \Big\lVert{\textstyle\sum_{i=1}^n} \bm{\phi}_i\Big\rVert
\geq
\lambda\sqrt{{\textstyle\sum_{i=1}^n}\sigma_i^2  } \bigg]
\leq
2\exp(-\lambda^2/4),
\end{equation}
which proves the claim for case~(I). Note that if $d=1$, we can directly obtain
\begin{equation}
\mathbb{P}\bigg[
\Big\lVert{\textstyle\sum_{i=1}^n}\bm{\phi}_i\Big\rVert
\geq
	\lambda\sqrt{{\textstyle\sum_{i=1}^n}\sigma_i^2} \bigg]
	\leq 2\exp(-\lambda^2/3)
\end{equation}
from inequality~\eqref{eq:martingale_total_expectation} and the implication (4)$\implies$(1)
in Theorem~\ref{Thm:equivalence}.\par
%%%%%%%%%%
\vspace{3pt}
%%%%%%%%%%
{\bf Case (II)} (from Ref.~\cite{Jin2019}, with corrections):
Consider a random vector $\bm{\phi}\in\mathbb{R}^d$ satisfying
$\mathbb{E}\left[ \bm{\phi} \right] =\bm{0}$ and
$\mathbb{E}\big[\exp(\norm{\bm{\phi}}^2/\sigma^2)\big]\leq \exp(1)$.
First, we notice that the real symmetric matrix
\begin{equation}\label{eq:martingale_proof_basic}
\mathbf{\Phi} \coloneqq
\begin{bmatrix}
	\,0         & \;\bm{\phi}^\top \\
	\,\bm{\phi} & \mathbf{0}
\end{bmatrix}\in\mathbb{R}^{(d+1)\times(d+1)}
\end{equation}
has a rank of at most $2$ and has eigenvalues 0 (with multiplicity $d-1$) and $\pm\norm{\bm{\phi}}$.
Letting $\mathbf{A}\preceq\mathbf{B}$ denote that $\mathbf{A}-\mathbf{B}$ is negative semidefinite,
we have
\begin{equation}\nonumber
\begin{split}
\mathbb{E}\left[\exp(\lambda\mathbf{\Phi})\right]
&=
\mathbb{E}\left[\exp(\lambda\mathbf{P}\mathbf{D}\mathbf{P}^{-1})\right]
=
\mathbb{E}\left[\mathbf{P}\exp(\lambda\mathbf{D})\mathbf{P}^{-1}\right]\\
&\preceq
\mathbb{E}\left[
\lambda\mathbf{P}\mathbf{D}\mathbf{P}^{-1} +
\mathbf{P}\exp\big(\tfrac{9}{16}\lambda^2 \mathbf{D}^2\big)\mathbf{P}^{-1}\right]
=
\lambda\mathbb{E}\left[\mathbf{\Phi}\right] + 
\mathbb{E}\left[
\mathbf{P}\exp\big(\tfrac{9}{16}\lambda^2 \mathbf{D}^2\big)\mathbf{P}^{-1} \right]\\
&\preceq
\mathbb{E}\big[ \mathbf{P} \exp(\tfrac{9}{16}\lambda^2 \norm{\bm{\phi}}^2)\mathbf{I}
\mathbf{P}^{-1} \big]
=
\mathbb{E}\big[\exp(\tfrac{9}{16}\lambda^2 \norm{\bm{\phi}}^2)\big]\mathbf{I}\\
&\preceq
\exp\left( \tfrac{9}{16}\sigma^2 \lambda^2\right)\mathbf{I}
\end{split}
\end{equation}
for all $\lambda$ such that $\lvert\lambda\rvert\sigma\leq 4/3$, where
$\mathbf{\Phi} =\mathbf{P}\mathbf{D}\mathbf{P}^{-1}$ is the eigendecomposition of $\mathbf{\Phi}$
($\mathbf{P}$ being an orthogonal matrix and $\mathbf{D}$ diagonal). We also have
\begin{equation}\nonumber
\begin{split}
\mathbb{E}\left[ \exp(\lambda \mathbf{\Phi}) \right]
=
\mathbb{E}\left[ \mathbf{P}\exp(\lambda \mathbf{D})\mathbf{P}^{-1}\right]
&\preceq
\mathbb{E}\left[ \mathbf{P}\exp(\norm{\lambda\bm{\phi}})\mathbf{I}\mathbf{P}^{-1}\right]
=
\mathbb{E}\left[ \exp(\norm{\lambda\bm{\phi}})\right]\mathbf{I}\\
&\preceq
\mathbb{E}\big[ \exp( x\lambda^2\sigma^2/4 + x^{-1} \norm{\bm{\phi}}^2/\sigma^2) \big]
\mathbf{I}\\
&\preceq
\exp( x\lambda^2\sigma^2/4 + 1/x)\mathbf{I}
\end{split}
\end{equation}
for all $x\geq 1$. Thus, the argument in the proof for case~(a) of Theorem~\ref{Thm:centered} also
applies here, leading to
\begin{equation}\label{eq:martingale_proof_1}
\mathbb{E}\left[ \exp(\lambda \mathbf{\Phi}) \right]
\preceq
\exp\big(\tfrac{3}{4} \sigma^2\lambda^2\big)\mathbf{I}
\end{equation}
for all $\lambda\in\mathbb{R}$.
Second, for real symmetric matrices $\mathbf{A}$, $\mathbf{B}$, and $\mathbf{C}$, we have
\begin{equation}\label{eq:martingale_proof_2}
	\operatorname{tr}\exp(\mathbf{A}+ \mathbf{B}) \leq
	\operatorname{tr}\big(\exp(\mathbf{A})\exp(\mathbf{B})\big),
\end{equation}
which is known as the Golden--Thompson inequality~\cite{Golden1965,Thompson1965}, and it is easy
to verify that
\begin{equation}\label{eq:martingale_proof_3}
	\operatorname{tr}(\mathbf{C}\mathbf{A})\leq
	\operatorname{tr}(\mathbf{C}\mathbf{B})
\end{equation}
if $\mathbf{A} \preceq \mathbf{B}$ and $\mathbf{C}\succeq \mathbf{0}$.
By defining $\mathbf{\Phi}_i$ with $\bm{\phi}_i$ ($i=1,2,3,\dots,n$) according to
Eq.~\eqref{eq:martingale_proof_basic} and collecting all preparatory results, we find
\begin{equation}\nonumber
\begin{split}
\mathbb{E}\Big[
\operatorname{tr}\exp\big(\lambda{\textstyle\sum^n_{i=1}}\mathbf{\Phi}_i\big)
\Big]
&=	\mathbb{E}\Big[\mathbb{E}\Big[
	\operatorname{tr}
	\exp(\lambda{\textstyle\sum^{n-1}_{i=1}}\mathbf{\Phi}_i + \lambda\mathbf{\Phi}_n)
	\Big\vert\mathcal{F}_{n-1}\Big]\Big]\\
&\leq\mathbb{E}\Big[\mathbb{E}\Big[\operatorname{tr}\big(
	\exp(\lambda{\textstyle\sum^{n-1}_{i=1}}\mathbf{\Phi}_i)
	\exp(\lambda\mathbf{\Phi}_n)\big)
	\Big\vert\mathcal{F}_{n-1}\Big]\Big]\\
&=  \mathbb{E}\Big[ \operatorname{tr}\Big(
	\exp(\lambda{\textstyle\sum^{n-1}_{i=1}}\mathbf{\Phi}_i)\cdot
	\mathbb{E}\big[
	\exp(\lambda\mathbf{\Phi}_n) \big\vert\mathcal{F}_{n-1} \big]
	\Big)\Big]\\
%%%%%
&\leq\mathbb{E}\Big[ \operatorname{tr}\Big(
	\exp(\lambda{\textstyle\sum^{n-1}_{i=1}}\mathbf{\Phi}_i)\cdot
	\exp\big(\tfrac{3}{4}\sigma_n^2 \lambda^2\big)\,\mathbf{I}
	\Big)\Big]\\
&=  \exp\big(\tfrac{3}{4}\sigma_n^2 \lambda^2\big)
 	\mathbb{E}\Big[\operatorname{tr}
 	\exp\big(\lambda{\textstyle\sum^{n-1}_{i=1}}\mathbf{\Phi}_i\big)
 	\Big]
 	\leq \cdots\\
%%%%%
&\leq\exp\big(\tfrac{3}{4} \lambda^2
	{\textstyle\sum^{n}_{i=1}}\sigma_i^2\big)\operatorname{tr}\mathbf{I}\\
&= (d+1)\exp\big(\tfrac{3}{4}\lambda^2 {\textstyle\sum^{n}_{i=1}}\sigma_i^2\big)
\end{split}
\end{equation}
for any $\lambda\in\mathbb{R}$,
where we applied expression~\eqref{eq:martingale_proof_2} for the first sign of inequality, and
used~\eqref{eq:martingale_proof_1} and~\eqref{eq:martingale_proof_3} for the second.
Note that
\begin{equation}\nonumber
\operatorname{tr}\exp\big(\lambda{\textstyle\sum^n_{i=1}}\mathbf{\Phi}_i\big)
=
\exp( \lambda\norm{\textstyle\sum^n_{i=1}\bm{\phi}_i}) +
\exp(-\lambda\norm{\textstyle\sum^n_{i=1}\bm{\phi}_i}) +(d-1)
\end{equation}
for any $\lambda\in\mathbb{R}$; therefore, we have
\begin{equation}
\mathbb{E}\Big[ \exp\big(x\norm{{\textstyle\sum^{n}_{i=1}}\bm{\phi}_i}\big) \Big]
\leq
\big(d+1\big)
\exp\Big(\tfrac{3}{4}x^2{\textstyle\sum^n_{i=1}}\sigma_i^2\Big)
\end{equation}
for any $x\in\mathbb{R}$. This finally leads to
\begin{equation}
\mathbb{P}\bigg[
\Big\lVert{\textstyle\sum^n_{i=1}} \bm{\phi}_i\Big\rVert
\geq
\lambda\sqrt{{\textstyle\sum^n_{i=1}}\sigma_i^2} \bigg]
%%%%%
\leq
(d+1)\exp(-\lambda^2/3)
\end{equation}
for all $\lambda\geq 0$, according to the implication (4)$\implies$(1)
in Theorem~\ref{Thm:equivalence} and Remark~\ref{Rmk:subgaussian2}.\par
%%%%%%%%%%
\vspace{3pt}
%%%%%%%%%%
{\bf Case (III)}:
We notice that the Euclidean unit ball in $\mathbb{R}^d$ can be covered by $(2/\epsilon+1)^d$
Euclidean balls of radius $\epsilon$ centered within the unit ball, for any $\epsilon\in(0,1)$
(check Corollary 4.2.13 in~\cite{Vershynin2018}; Example 5.8 in~\cite{Wainwright2019}), and let
$\bm{c}_k$ denote the centers of these balls in such a cover
($k=1,2,\dots, \operatorname{floor}\big((1+2/\epsilon)^d\big)$).
Then for any given unit vector $\bm{e}_\textup{u}\in\mathbb{S}^{d-1}\mkern-3mu\subset\mathbb{R}^d$
there exists $k$ such that $\norm{\bm{e}_\textup{u} - \bm{c}_k}\leq\epsilon$, and we have
\begin{equation}\nonumber
\begin{split}
\norm{{\textstyle\sum^n_{i=1}}\bm{\phi}_i }
=
\max_{ \norm{\bm{e}_\textup{u}}=1 }
\left\{ \bm{e}_\textup{u}^\top{\textstyle\sum^n_{i=1}}\bm{\phi}_i \right\}
&\leq
\max_{ 1\leq k\leq (1+2/\epsilon)^d }
\left\{ \bm{c}_k^\top{\textstyle\sum^n_{i=1}}\bm{\phi}_i \right\}+
\max_{ \norm{\bm{e}}\leq\epsilon}
\left\{ \bm{e}^\top {\textstyle\sum^n_{i=1}}\bm{\phi}_i  \right\}\\
&\leq
\max_{ 1\leq k\leq (1+2/\epsilon)^d }
\left\{\bm{c}_k^\top{\textstyle\sum^n_{i=1}}\bm{\phi}_i  \right\} +
\epsilon\norm{{\textstyle\sum^n_{i=1}}\bm{\phi}_i },
\end{split}
\end{equation}
which indicates
\begin{equation}
0\leq \norm{{\textstyle\sum^n_{i=1}}\bm{\phi}_i}
\leq (1-\epsilon)^{-1}
\max_{ 1\leq k\leq (1+2/\epsilon)^d }
\left\{\bm{c}_k^\top {\textstyle\sum^n_{i=1}}\bm{\phi}_i \right\}
\end{equation}
for all $\epsilon\in(0,1)$.
Now, we examine the subgaussianity of $\bm{c}_k^\top\sum_{i=1}^n\bm{\phi}_i$.
Given the assumptions
$\mathbb{E}\left[ \exp\big((\bm{e}_\textup{u}^\top\bm{\phi}_i)^2/\sigma_i^2\big)\big\vert
\mathcal{F}_{i-1}\right] \leq \exp(1)$
and
$\mathbb{E}\left[ \bm{e}_\textup{u}^\top\bm{\phi}_i \big\vert \mathcal{F}_{i-1}\right]
= \bm{e}_\textup{u}^\top\mathbb{E}\left[ \bm{\phi}_i \big\vert \mathcal{F}_{i-1}\right] = 0$
for any unit vector $\bm{e}_\textup{u}$, Theorem~\ref{Thm:centered} indicates that
$ \bm{e}_\textup{u}^\top\bm{\phi}_i $, conditional on $\mathcal{F}_{i-1}$, is
$(3\sigma_i^2/2, 1)$-subgaussian, and 
\begin{equation}\nonumber
\begin{split}
\mathbb{E}\big[\exp(\lambda\bm{c}_k^\top{\textstyle\sum^n_{i=1}}\bm{\phi}_i) \big]
&=
\mathbb{E}\Big[\exp(\lambda\bm{c}_k^\top{\textstyle\sum^{n-1}_{i=1}}\bm{\phi}_i)\cdot
\mathbb{E}\big[\exp(\lambda\bm{c}_k^\top\bm{\phi}_n)\vert \mathcal{F}_{n-1} \big]\Big]\\
&\leq
\exp\big( \tfrac{3}{4}\norm{\bm{c}_k}^2\sigma_n^2\lambda^2 \big)
\mathbb{E}\big[\exp(\lambda\bm{c}_k^\top{\textstyle\sum^{n-1}_{i=1}}\bm{\phi}_i)\big]
\leq \dots\\
&\leq
\exp\Big[ \tfrac{3}{4}\lambda^2\norm{\bm{c}_k}^2{\textstyle\sum^n_{i=1}}\sigma_i^2 \Big]
\end{split}
\end{equation}
for all $\lambda\in\mathbb{R}$ and all $\bm{c}_k$.
Now, according to case~(iii) of Theorem~\ref{Thm:closure1},
\begin{equation}\nonumber
\begin{split}
\mathbb{E}\big[\exp(\lambda \norm{{\textstyle\sum^n_{i=1}}\bm{\phi}_i}) \big]
&\leq
\mathbb{E}\left[\exp\left( \frac{\lvert\lambda\rvert}{1-\epsilon}
{\max_k} \Big\{\bm{c}_k^\top {\textstyle\sum^n_{i=1}}\bm{\phi}_i \Big\}\right)
\right]\\
&\leq
\left(1+\frac{2}{\epsilon}\right)^d
\exp\left( \frac{\lambda^2}{(1-\epsilon)^2} 
\cdot\frac{3}{4}{\max_k}\big\{\norm{\bm{c}_k}^2\big\}  
\cdot{\textstyle\sum^n_{i=1}}\sigma_i^2\right)\\
&\leq
\left(1+\frac{2}{\epsilon}\right)^d
\exp\left( \frac{1}{(1-\epsilon)^2} 
\cdot\frac{3}{4}\lambda^2
\cdot{\textstyle\sum^n_{i=1}}\sigma_i^2\right)
\end{split}
\end{equation}
for all $\lambda\in\mathbb{R}$. Finally, we obtain
\begin{equation}
\mathbb{P}\bigg[
\Big\lVert{\textstyle\sum^n_{i=1}}\bm{\phi}_i\Big\rVert
\geq \frac{\lambda}{1-\epsilon}\sqrt{\tfrac{3}{2}{\textstyle\sum^n_{i=1}}\sigma_i^2}
\bigg]
\leq
\Big(1+\frac{2}{\epsilon}\Big)^d\exp(-\lambda^2/2),
\end{equation}
or equivalently
\begin{equation}\label{eq:martingale_case_III}
\mathbb{P}\bigg[
\Big\lVert{\textstyle\sum^n_{i=1}}\bm{\phi}_i\Big\rVert
\geq \lambda\sqrt{{\textstyle\sum^n_{i=1}}\sigma_i^2}\bigg]
\leq
\Big(1+\frac{2}{\epsilon}\Big)^d
\exp\left(-\frac{1}{3}\big(1-\epsilon\big)^2 \lambda^2 \right),
\end{equation}
for all $\lambda\geq 0$ and $\epsilon\in(0,1)$. By choosing $\epsilon = 1/2$ we obtain
\begin{equation}
\mathbb{P}\bigg[
\Big\lVert{\textstyle\sum^n_{i=1}}\bm{\phi}_i\Big\rVert
\geq \lambda\sqrt{{\textstyle\sum^n_{i=1}}\sigma_i^2}\bigg]
\leq
5^d \exp\left(-\lambda^2/12\right)
\end{equation}
for all $\lambda\geq 0$.\par
%%%%%%%%%%
\pagebreak
%\vspace{3pt}
%%%%%%%%%%
The assumptions (I), (II), and (III) are progressively weaker, in the sense that
(I) implies (II) according to (ii) in Theorem~\ref{Thm:closure2}, and
(II) implies (III) as $(\bm{e}_\textup{u}^\top\bm{\phi}_i)^2\leq \norm{\bm{e}_\textup{u}}^2
\norm{\bm{\phi}_i}^2 = \norm{\bm{\phi}_i}^2$.
With $\mathbb{E}\left[ \bm{e}_\textup{u}^\top\bm{\phi}_i \vert\mathcal{F}_{i-1} \right]=0$, we have
$\mathbb{E}\big[ \exp(\lambda\bm{e}_\textup{u}^\top\bm{\phi}_i )
\vert \mathcal{F}_{i-1}\big] \leq \exp\big( 3\sigma_i^2 \lambda^2/4\big) $
for any unit vector $\bm{e}_\textup{u}\in\mathbb{R}^d$ and $i=1,2,3,\dots$, and thus
\begin{equation}\nonumber
\begin{split}
\mathbb{E}\big[\exp(\lambda\bm{e}_\textup{u}^\top{\textstyle\sum^n_{i=1}}\bm{\phi}_i)\big]
&=\mathbb{E}\Big[
	\exp(\lambda\bm{e}_\textup{u}^\top{\textstyle\sum^{n-1}_{i=1}}\bm{\phi}_i)\cdot
	\mathbb{E}\big[\exp(\lambda\bm{e}_\textup{u}^\top\bm{\phi}_n)\vert\mathcal{F}_{n-1}\big]\Big]\\
&\leq
	\exp\big( \tfrac{3}{4}\sigma_n^2\lambda^2\big)
	\mathbb{E}\big[\exp(\lambda{\textstyle\sum^{n-1}_{i=1}}\bm{\phi}_i)\big]
	\leq \cdots\\
&\leq\exp\Big(\tfrac{3}{4}\lambda^2{\textstyle\sum^n_{i=1}}\sigma_i^2 \Big)
\end{split}
\end{equation}
for all $\lambda\in\mathbb{R}$. Therefore, we have
\begin{equation}\nonumber
\mathbb{P}\bigg[
\bm{e}_\textup{u}^\top{\textstyle\sum^n_{i=1}\bm{\phi_i}}
	\geq \lambda\sqrt{{\textstyle\sum^n_{i=1}}\sigma_i^2} \bigg]
\leq
\exp(-\lambda^2/3)
\end{equation}
for any unit vector $\bm{e}_\textup{u}\in\mathbb{R}^d$ and $\lambda\geq 0$,
according to the implication (4)$\implies$(5) in Theorem~\ref{Thm:equivalence}
and Remark~\ref{Rmk:subgaussian1}.
\end{proof}

\begin{remark}
Smaller values of $\epsilon$ in inequality~\eqref{eq:martingale_case_III} can be chosen
when the variance proxy is prioritized. For example, by choosing $\epsilon=1/3$, we have
\begin{equation}\nonumber
\mathbb{P}\bigg[
	\big\lVert{\textstyle\sum^n_{i=1}}\bm{\phi}_i\big\rVert
	\geq \lambda\sqrt{{\textstyle\sum^n_{i=1}}\sigma_i^2}\bigg]
	\leq
	7^d \exp\left(-\tfrac{4}{27}\lambda^2\right)
	\leq
	7^d \exp\left(-\lambda^2/7\right)
\end{equation}
for all $\lambda\geq 0$.
\end{remark}

\begin{remark}
When subgaussianity of $\norm{\textstyle\sum^n_{i=1}\bm{\phi}_i}$ is described by
$\mathbb{E}[\exp(\norm{\textstyle\sum^n_{i=1}\bm{\phi}_i}^2\mkern-3mu/\sigma^2 )] \leq \rho$
with $\rho$ a constant, we will see
\begin{equation}\nonumber
\begin{split}
	&\frac{\sigma^2}{\textstyle\sum^n_{i=1}\sigma_i^2}
	= \left\{
	\begin{aligned}
		&\mathcal{O}(1),        &&\textup{if (I) holds;}\\[1pt] 
		&\mathcal{O}(\ln(d+1)), &&\textup{if (II) holds;}\\[1pt]
		&\mathcal{O}(d),        &&\textup{if (III) holds,}
	\end{aligned}
	\right.
\end{split}
\end{equation}
with details omitted here.
\end{remark}

\section{Summary\label{sec:summary}}
This note restates the most basic properties of subgaussian random variables using the definition
of $(\sigma, \rho)$-subgaussianity.
It includes the equivalence of various characterizations, the treatment of (centered)
$(\sigma, 1)$-subgaussians, the closure of subgaussianity under simple operations, and an
application to the large deviation analysis of subgaussian vector martingale differences. 
While this note adheres to well-established results,
it also provides flexibility in translating between and aligning these results,
as well as in manipulating their details.

\acks{The author acknowledges the support from
the Institute of AI and Beyond of the University of Tokyo,
and JST Moonshot R\&D Grant Number JPMJMS2021.}


%\newpage

%\appendix
%\section{}%\label{app:theorem}
% Note: in this sample, the section number is hard-coded in. Following
% proper LaTeX conventions, it should properly be coded as a reference:

%In this appendix we prove the following theorem from
%Section~\ref{sec:textree-generalization}:
%\section{}

\vskip 0.2in
%\nocite{*}
%\bibliographystyle{unsrtnat}
\bibliography{Subgaussian}
\end{document}